\documentclass[sn-mathphys-num]{sn-jnl}

\usepackage{graphicx}%
\usepackage{multirow}%
\usepackage{amsmath,amssymb,amsfonts}%
\usepackage{amsthm}%
\usepackage{mathrsfs}%
\usepackage[title]{appendix}%
\usepackage{xcolor}%
\usepackage{textcomp}%
\usepackage{manyfoot}%
\usepackage{booktabs}%
\usepackage{algorithm}%
\usepackage{algorithmicx}%
\usepackage{algpseudocode}%
\usepackage{listings}%
\usepackage{format}

\begin{document}

\title[$1$-dim reps of $W$-algebra]{On invariant $1$-dimensional representations of a finite $W$-algebra}

\author*{\fnm{Dmytro} \sur{Matvieievskyi}}\email{dmitriy.matveevskiy@gmail.com}

\affil{\orgname{University of Massachusetts Amherst}, \orgaddress{\city{Amherst}, \state{MA}, \country{USA}}}

	\abstract{
		Let $\fg$ be a simple Lie algebra over $\CC$ and $G$ be the corresponding simply connected algebraic group. Consider a nilpotent element $e\in \fg$, the corresponding element $\chi=(e, \bullet)$ in $\fg^*$, and the coadjoint orbit $\OO=G\chi$. We are interested in the set $\fJ\fd^1(\cW)$ of codimension $1$ ideals $J\subset \cW$ in a finite $W$-algebra $\cW=U(\fg, e)$. We have a natural action of the component group $\Gamma=Z_G(\chi)/Z_G^\circ(\chi)$ on $\fJ\fd^1(\cW)$. Denote the set of $\Gamma$-stable points of $\fJ\fd^1(\cW)$ by $\fJ\fd^{1}(\cW)^{\Gamma}$. For a classical $\fg$ Premet and Topley proved that $\fJ\fd^{1}(\cW)^{\Gamma}$ are parameterized by the points of an affine space. In this paper we will give an alternative proof of this fact using completely different approach. As a corollary of our construction, we imply that for classical $\fg$ all $\Gamma$-invariant $1$-dimensional representations of $\cW$ are obtained by a parabolic induction introduced by Losev.
  }



 \maketitle
 
	\section{Introduction}
	%
	%
	\subsection{}
	Let $\fg$ be a simple Lie algebra over $\CC$ and $G$ be the corresponding simply connected algebraic group. Let $e\in \fg$ be a nilpotent element and $\chi=(e,\bullet)$ be the dual element in $\fg^*$, where $(\bullet, \bullet)$ stands for the Killing form. We fix the coadjoint orbit $\OO=G\chi$. In \cite{Premet2002} Premet defined the finite $W$-algebra $\cW=U(\fg,e)$ associated to the nilpotent element $e$. In \cite{Losev3} Losev gave an alternative definition that will be recalled in \cref{W-alg}. 
	
	Let $\fJ\fd(\cW)$ and $\fJ\fd^1(\cW)$ be the sets of ideals and codimension $1$ ideals in $\cW$ respectively.
	Note that every codimension $1$ ideal in $\cW$ contains the two-sided ideal $([\cW,\cW])$ generated by the commutator $[\cW,\cW]$. Therefore the set $\fJ\fd^1(\cW)$ is isomorphic to the set of codimension $1$ ideals in the commutative algebra $\cW_{ab}:=\cW/([\cW,\cW])$. By the Hilbert Nullstellensatz such ideals are classified by the points of the maximal spectrum $\Specm(\cW_{ab})$. 
	
	We set $\Gamma:=Z_G(e)/Z_G^\circ(e)$ to be the component group of the centralizer $Z_G(e)$ of $e$ and $Q$ to be the reductive part of $Z_G(e)$. The action of $Q$ on $\cW$ preserves the ideal $([\cW, \cW])$ and induces an action on $\Specm(\cW_{ab})$. Recall that $\Gamma$ is also isomorphic to $Q/Q^\circ$, and the action of $Q^\circ$ on any $J\in \fJ\fd^1(\cW)$ is trivial because of the embedding $\fq \to \cW$, see \cite[Section 2]{Premetstabilizer} for more details. Therefore we have an action of $\Gamma$ on $\Specm(\cW_{ab})$. Let $\fJ\fd^1(\cW)^\Gamma$ be the set of $\Gamma$-invariant codimension $1$ ideals in $\cW$.
    For a Levi subalgebra $\fl\subset \fg$ let $\fp$ be a parabolic subalgebra with the semisimple part $\fl$. Let $\OO'\subset \fl^*$ be a nilpotent orbit. Following \cite{Losev4}, we say that the orbit $\OO$ is \emph{birationally induced} from $\OO'$ if the generalized Springer map $G\times^P (\overline{\OO}'\times \fp^\perp)\to \fg^*$ factors through the embedding $\overline{\OO}\to \fg^*$, and the map $G\times^P (\overline{\OO}'\times \fp^\perp)\to \overline{\OO}$ is birational. An orbit is called \emph{birationally rigid} if it cannot be birationally induced using any proper Levi subgroup, and for any nilpotent orbit $\OO$ there is a unique up to conjugation pair of a Levi $\fl$ and a birationally rigid nilpotent orbit $\OO_0$, such that $\OO$ is birationally induced from $\OO_0$. We discuss it in more detail in \cref{quant affin}.
	The main goal of this paper is to give a geometric proof of the following theorem. 
	\begin{theorem}\label{theorem}
		Let $\fg$ be classical, and let $(\fl, \OO_0)$ be as above. There is a finite group $W$ acting on $(\fl/[\fl,\fl])^*$ by reflections, such that $\fJ\fd^1(\cW)^\Gamma$ are parameterized by the points of $(\fl/[\fl,\fl])^*/W$. Therefore, $\fJ\fd^1(\cW)^\Gamma$ are parameterized by the points of an affine space.
	\end{theorem}
	We will define the group $W$ in \cref{quant affin} and give a more precise statement of the theorem in \cref{theorem_ref}. The theorem above is inspired by \cite[Theorem 2]{Premet-Topley}. In \emph{loc.cit.}, a stronger result was shown, namely that $\cW_{ab}$ is isomorphic to the polynomial algebra. However, our strategy of the proof is very different to \cite{Premet-Topley} and allows an explicit Lie-theoretic description of the polynomial algebra (or in our case the underlying affine space).  
	
	In \cite{Losev_parabolic} Losev introduced a parabolic induction for finite-dimensional representations of $W$-algebras. In Section 6 we give a geometric interpretation of the parabolic induction for $\Gamma$-invariant $1$-dimensional representations of $W$-algebras. Let us briefly recall the objects in Losev's construction. Let $\fg$ be any simple Lie algebra, and let $\OO\subset \fg^*$ be a nilpotent orbit birationally induced from some nilpotent orbit $\OO_0\subset \fl^*$ for a Levi subalgebra $\fl\subset\fg$. Let $\cW$ and $\cW_0$ be the two $W$-algebras corresponding to $\OO$ and $\OO_0$ respectively, and $\Gamma$ and $\Gamma_0$ be the corresponding component groups. Set $\fJ\fd^{fin}(\cW)$ and $\fJ\fd^{fin}(\cW_0)$ be the sets of the ideal of finite codimension in $\cW$ and $\cW_0$ correspondingly. Then Losev constructed a map $\rho: \fJ\fd^{fin}(\cW_0)\to \fJ\fd^{fin}(\cW)$. The key content of Section 6 is  a geometric description of this map on the subset $\fJ\fd^1(\cW_0)^{\Gamma_0}$. For classical $\fg$ we prove that the restriction of $\rho$ to the subset $\fJ\fd^1(\cW_0)^{\Gamma_0}$ gives a map $\rho^\Gamma: \fJ\fd^1(\cW_0)^{\Gamma_0}\to \fJ\fd^1(\cW)^\Gamma$, and the map $\rho^\Gamma$ is surjective. As a corollary with earlier results, we prove the following statement. 
	\begin{theorem}\label{theorem2}
		Let $\fg$ be classical. Suppose that the orbit $\OO$ is not birationally rigid. Then all $\Gamma$-invariant $1$-dimensional representations of $\cW$ are induced from a proper Levi algebra $\fl\subset \fg$.
	\end{theorem}
    Since every rigid orbit is birationally rigid, this theorem is a formally weaker version of \cite[Theorem 5]{Premet-Topley}. In fact, one can show that these two theorems are equivalent, but this won't be pursued in the present paper. Similar to \cref{theorem} and \cite[Theorem 2]{Premet-Topley}, the methods used in proofs are completely different. 
	\begin{rmk}
		The results of \cite{Brown2018} imply that for some exceptional Lie algebras $\fg$ both \cref{theorem} and \cref{theorem2} do not hold. Consider the orbit $\OO=A_3+A_1$ in the Lie algebra $\fg$ of type $E_6$. By \cite[Table 6]{deGraafElashvili}, $\OO$ is induced from the orbit $\OO_0=(3, 2, 2, 1, 1, 1)$ in Levi $\fl$ of type $D_5$, and $\fl$ is unique proper Levi subalgebra such that $\OO$ can be induced from $\fl$. Since $\Gamma=1$ \cite[Section 8.4]{CM}, it implies that the induction above is birational. As shown in \cite{Brown2018}, $\fJ\fd^1(\cW)^\Gamma = \fJ\fd^1(\cW) \simeq \CC \sqcup \{pt\}$, which contradicts the statement of \cref{theorem}. Note that $\fl_0$ is a corank $1$ Levi subalgebra in $E_6$, and $\Gamma_0=1$. Then by \cref{theorem} there is a unique codimension $1$ ideal in $\cW_0$. Thus, the $1$-dimensional representations of $\cW$ induced from $\fl$ are parameterized by elements of $(\fl/[\fl,\fl])^*\simeq \CC$. Therefore, \cref{theorem2} does not hold for the orbit $\OO$.
	\end{rmk}
	\begin{rmk}
	    We want to draw attention to an exciting recent paper \cite{GTW} of Goodwin, Topley and Westaway. Let $\fg$ be a classical Lie algebra, and $\OO$ a nilpotent not rigid orbit. In \emph{loc.cit.} the authors prove that \emph{any} $1$-dimensional representation of a finite $W$-algebra $\cW$ attached to $\OO$ is induced from a proper Levi $\fl\subset \fg$. Using \cite[Theorem 9.9]{Topley} as a replacement of \cref{extension}, the methods of \cref{para section} can be used to give a geometric proof of this theorem. Namely, the geometric construction of the parabolic induction of $\Gamma$-invariant $1$-dimensional representations of $W$-algebras can be adapted to all $1$-dimensional representations of $W$-algebras. As this would require working with $G$-equivariant covers of nilpotent orbits, while this paper considers only nilpotent orbits, we do not give any details.
	\end{rmk}
	\subsection{}
	The plan of the paper is as follows. It is known \cite{Losev2} that there is a natural bijection between the set of $\Gamma$-invariant ideals $J\in \fJ\fd^1(\cW)^{\Gamma}$ and the space $\cQ(\OO)$ of formal graded $G$-equivariant quantizations of $\OO$. We want to relate the sets of the formal graded quantizations of the coadjoint orbit $\OO$ and of its affinization $\Spec(\CC[\OO])$. The set of formal graded quantizations of the affinization was computed by Losev in \cite{Losev4}. In \cref{affinization} we consider the isomorphism classes of formal graded quantizations of $\Spec(\CC[\OO])$. By \cite{Losev4} these are classified by the points of $\fP/W$ for some affine space $\fP$ and a finite group $W$. We follow loc.cit. to give a description of $\fP$ and $W$ using Lusztig-Spaltenstein induction. In \cref{quant orbit} following \cite{Losev2} we construct a natural bijection between the set of $\Gamma$-invariant ideals $J\in \fJ\fd^1(\cW)^{\Gamma}$ and the space $\cQ(\OO)$ of formal $G$-equivariant quantizations of $\OO$ mentioned above. In Sections \ref{analytification of quantization} and \ref{lift} we show that $\cQ(\OO)$ and $\fP/W$ are in a natural bijection. This is the main part of this paper. The construction of Section 3 gives us an algebraic quantization of an affine subscheme of $\fg^*$ with the reduced scheme $\overline{\OO}$. We show that it can be extended to a holomorphic one (i.e. a quantization in the complex-analytic topology). After that, we lift that holomorphic quantization to a quantization of $\Spec(\CC[\OO])$. The key difference here is that in complex-analytic topology a neighborhood of a point $x$ of a codimension $2$ leaf of $\Spec(\CC[\OO])$ can be written as a product of a Kleinian singularity and an open disk, which allows us to consider the restriction of a holomorphic quantization to this neighborhood. That gives us a bijection between $\cQ(\OO)$ and $\fP/W$ as sets. We deduce \cref{theorem} from that.
	
	In \cref{para section} we define the parabolic induction for quantizations of nilpotent orbits, that mirrors the definition of the parabolic induction for $1$-dimensional representations of $W$-algebras from \cite{Losev_parabolic}. Then we use this definition to prove \cref{theorem2}.

	\section{Quantizations of $\Spec(\CC[\OO])$}\label{affinization}
	\subsection{Graded formal quantizations}\label{quantizations}
	Let $A$ be a finitely generated Poisson algebra equipped with a grading $A=\bigoplus_{i=0}^\infty A_i$ such that $A_0=\CC$, and the Poisson bracket has degree $-d$, where $d\in \ZZ_{>0}$. 
	%
	
	Let $\cA_\hbar$ be a $\CC[[\hbar]]$-algebra, flat over $\CC[[\hbar]]$ and complete and separated in the $\hbar$-adic topology. Assume that $\cA_\hbar$ is equipped with a $\CC^\times$-action by $\CC$-algebra automorphisms that is rational on all quotients $\cA_\hbar/(\hbar^k)$ and satisfies $t.\hbar=t\hbar$. Moreover, suppose that $[\cA_\hbar, \cA_\hbar]\subset \hbar^d \cA_\hbar$. All $\CC[[\hbar]]$-algebras in this paper are assumed to satisfy all of these conditions. We have a skew-symmetric bracket on $\cA_\hbar$ given by $\{a,b\}=\frac{[a,b]}{\hbar^d}$ that induces a Poisson bracket on $\cA_\hbar/(\hbar)$. The action of the torus on $\cA_\hbar$ induces a grading on the quotient $\cA_\hbar/(\hbar)$. Suppose that there is an isomorphism $\theta: \cA_\hbar/(\hbar)\simeq A$ of graded Poisson algebras. We say that $\cA_\hbar$ together with an isomorphism $\theta$ is a graded formal quantization of $A$.  
	
	\begin{example}\label{UformS}
		Equip the Poisson algebra $S(\fg)$ with the grading by putting $\fg$ in degree $2$. The completed homogeneous universal enveloping algebra $U_\hbar(\fg)=T(\fg)[[\hbar]]/(\xi\otimes\eta-\eta\otimes\xi-\hbar^2[\xi, \eta]\mid \xi,\eta\in \fg)$ is a formal quantization of $S(\fg)$.
	\end{example}
	
	By a Poisson scheme we mean a scheme $X$ over $\Spec(\CC)$ whose structure sheaf $\cO_X$ is equipped with a Poisson bracket. For example, for any finitely generated commutative Poisson algebra $A$ the affine scheme $X=\Spec(A)$ is a Poisson scheme. By a formal quantization of $X$ we understand a sheaf $\calD$ of $\CC[[\hbar]]$-algebras such that $\calD$ is flat over $\CC[[\hbar]]$ and complete and separated in the $\hbar$-adic topology together with an isomorphism of sheaves of Poisson algebras $\theta: \calD/\hbar\calD \simeq \cO_X$. In what follows we assume that $X$ is conical. We introduce the conical topology on $X$, where, by definition, a subset $U$ is open if and only if $U$ is Zariski open and $\CC^\times$-stable. Note that if $X$ is normal then every point admits a $\CC^\times$-stable affine open neighborhood, see \cite[Corollary 2]{Sumihiro1974}. From now on we consider quantizations in the conical topology and demand that $\calD$ is endowed with a $\CC^\times$-action, such that the isomorphism $\theta:\calD/\hbar\calD \simeq \cO_X$ to be a $\CC^\times$-invariant isomorphism of sheaves of Poisson algebras. Such pair $(\calD, \theta)$ is called a \emph{graded formal quantization}. 
	
	\begin{rmk}\label{conical}
		The coherent sheaves on $X$ in the conical topology correspond to the graded finitely generated modules $M$ over $A$. In this paper we want to consider formal graded quantizations of $X$ instead of all formal quantizations of $X$, the two definitions are differed by imposing the $\CC^\times$-invariance condition in the end of the last paragraph. Therefore, we need to work in the conical topology. 
	\end{rmk}
	
	A morphism of two formal quantizations ($\calD_1$, $\theta_1$) and ($\calD_2$, $\theta_2$) is a $\CC^\times$-equivariant morphism $\phi: \calD_1\to \calD_2$ of $\CC[[\hbar]]$-algebras such that $\theta_1=\theta_2\circ\phi$. Any such morphism is automatically an isomorphism of formal quantizations.
	
	We will need the following standard lemma (see, for example, \cite[Lemma 2.10]{BezLosev}).
	\begin{lemma}\label{affinequant}
		Let $A$ be a Poisson algebra and $X=\Spec(A)$. The map $\calD\to \Gamma(X, \calD)$ gives a bijection between the set of formal quantizations of $X$ and the set of formal quantizations of $A$.
	\end{lemma}
	\begin{definition}\label{microloc}
		The inverse of the map in \cref{affinequant} produces a sheaf $\calD$ quantizing $X$ starting from a quantization $\cA_\hbar$ of an algebra $A$. Such sheaf $\calD$ is called a {\itshape microlocalization} of $\cA_\hbar$.
	\end{definition}
	
	Let $X$ be a (possibly singular) complex analytic space such that the sheaf $\calH_X$ of holomorphic functions is equipped with a Poisson bracket. We will call such $X$ a Poisson analytic space. Suppose that we have an action of $\CC^\times$ on $X$. Similarly to the notion of a formal quantization of $X$ we define a formal holomorphic quantization of $X$ to be a sheaf $\calD$ of $\CC[[\hbar]]$-algebras such that $\calD$ is flat over $\CC[[\hbar]]$ and complete and separated in the $\hbar$-adic topology together with a $\CC^\times$-invariant isomorphism of sheaves of Poisson algebras $\theta: \calD/\hbar\calD \simeq \calH_X$.
	
	\begin{rmk}\label{Rees}
		Filtered quantizations are often considered instead of graded formal quantizations. A filtered quantization of a Poisson algebra $A$ is a pair of a filtered associative unital algebra $\cA=\bigcup_{i\ge 0} \cA_{\le i}$, such that $[\cA_{\le k}, \cA_{\le l}]\subset \cA_{\le k+l-d}$, and an isomorphism of Poisson graded algebras $\Xi: \gr \cA:=\bigoplus_i \cA_{\le i}/\cA_{\le i-1}\simeq A$, where the Poisson bracket of degree $-d$ on $\gr \cA$ is induced from the commutator on $\cA$. 
		
		The sets of filtered quantizations and of graded formal quantizations of $A$ are in a bijection. Since we are implicitly using the construction of the bijection in this paper, let us recall it.
		
		Let $(\cA, \Xi)$ be a filtered quantization of $A$. We define the Rees algebra $R_\hbar(\cA)=\bigoplus_i \cA_{\le i}h^i$ of $\cA$ as a graded $\CC[\hbar]$-subalgebra in $\cA[h^{\pm 1}]$. There is a natural isomorphism of Poisson algebras $\phi: R_\hbar(\cA)/(\hbar)\to \gr \cA$. Let $R^{\wedge}_{\hbar}(\cA)$ be the $\hbar$-adic completion of $R_\hbar(\cA)$. The pair ($R^\wedge_\hbar(\cA), \Xi\circ \phi$) is a graded formal deformation of $A$.
		
		Now, let $(\cA_\hbar, \theta)$ be a graded formal quantization of $A$. We have a $\CC^\times$-action on $\cA_\hbar$, and set $\cA_{\hbar, fin}$ to be the linear span of all elements $a\in \cA_\hbar$ with $t.a=t^ia$ for some $i\in \ZZ$. It is easy to see that $\cA_{\hbar, fin}\subset \cA_\hbar$ is a graded subalgebra, and $\cA_\hbar/(\hbar)\simeq \cA_{\hbar, fin}/(\hbar)$. Set $\cA=\cA_{\hbar, fin}/(\hbar-1)$. Note that $\cA_{\hbar, fin}$ has a natural filtration coming from the grading, and therefore $\cA$ is a filtered algebra with an isomorphism $\phi: \gr\cA\simeq \cA_{\hbar, fin}/(\hbar)$. The pair $(\cA, \theta\circ \phi)$ is a filtered deformation of $A$, and the maps $\cA\to R^\wedge_\hbar(\cA)$ and $\cA_\hbar\to \cA_{\hbar, fin}/(\hbar-1)$ are the mutually inverse bijections between the sets of filtered quantizations and graded formal quantizations of $A$.
	\end{rmk}
	
	\begin{example}\label{diffop}
		Let $G$ be a complex Lie group. Let $\calD_G$ be a sheaf of differential operators on $G$. For any open $U$ we set $\calD_{G,\hbar}(U)=R^\wedge_{\hbar}(\calD_G(U))$. For every inclusion of open sets $V\subset U$ we have a restriction map $\calD_G(U)\to \calD_G(V)$ that extends to a map $R^\wedge_{\hbar}(\calD_G(U))\to R^\wedge_{\hbar}(\calD_G(V))$ on Rees completions. For any affine subset $U\subset G$ the algebra $\calD_{G,\hbar}(U)$ can be microlocalized to a formal quantization of $T^*U\subset T^*X$. Open subsets of $T^*U$ for $U$ affine define a basis of topology of $T^*G$, and thus this defines a quantization $\underline{\calD}_{G,\hbar}$ on the cotangent bundle $T^*G$.
		
		Recall from \cite{Kashiwara2008} the notion of a $W$-algebra on a symplectic manifold $X$. We note that despite the similarity in names, there is no immediate connection between $W$-algebras from \itshape{loc.cit.} and finite $W$-algebras considered in this paper. In \cite[Section 2.2.3]{Kashiwara2008} the canonical $W$-algebra $\cW_{T^*G}=\cup \cW_{T^*G}(i)$ coming from the sheaf of holomorphic differential operators on $G$ is constructed. The sheaf $\underline{\calD}_{G, hol,\hbar}=\cW_{T^*G}(0)$ is a formal holomorphic quantization of $T^*G$.
	\end{example}
	\begin{example}\label{differential operators}
		Set $X=\fg^*$ with the canonical structure of a Poisson variety. The left-invariant vector fields on $G$ give a trivialization of the cotangent bundle $T^*G\simeq G\times \fg^*$. Therefore we have a map $\pi: T^*G\to \fg^*$. We set $\calD:=(\pi_*\underline{\calD}_{G,\hbar})^G$ and $\calD_{hol}:=(\pi_*\underline{\calD}_{G,hol,\hbar})^G$. These are formal and formal holomorphic quantizations of $\fg^*$ correspondingly. The formal quantization of $S(\fg)$ corresponding to $\calD$ is $\Gamma(\fg^*, \calD)$. 
		
		By construction, we have $\Gamma(\fg^*, \calD)=\Gamma(G, \calD_{G,\hbar})^{G}=U_\hbar(\fg)$, so $\calD$ is a microlocalization of $U_\hbar(\fg)$. That is the formal analogue (in sense of \cref{Rees}) of the well-known fact that $\Gamma(G, \calD_G)^{G}=U(\fg)$, see, for example, \cite[Proposition 1.9]{Helgason}.
	\end{example}
	
	%
	%
	%
	\subsection{Quantizations of conical symplectic singularities}\label{quant of sing}
	Let $X$ be a normal Poisson algebraic variety such that the regular locus $X^{reg}$ is a symplectic variety. Let $\omega^{reg}$ be the symplectic form on $X^{reg}$. Suppose that there exists a projective resolution of singularities $\rho: \widehat{X}\to X$ such that $\rho^*(\omega^{reg})$ extends to a regular (not necessarily symplectic) form on $\widehat{X}$. Following Beauville \cite{Beauville2000}, we say that $X$ has symplectic singularities. We will be mostly interested in the following example.
	\begin{example}\cite{Panyushev1991}
		Let $\fg$ be a semisimple Lie algebra and $\OO\subset \fg^*$ be a nilpotent orbit. Then $X=\Spec(\CC[\OO])$ has symplectic singularities.
	\end{example}
	We say that an affine Poisson variety $X$ is conical if its algebra of functions is endowed with a grading $\CC[X]=\bigoplus_{i=0}^\infty \CC[X]_i$ and a positive integer $d$ such that $\CC[X]_0=\CC$, and for any $i,j$ and $f\in \CC[X]_i$, $g\in \CC[X]_j$ we have $\{f,g\}\in \CC[X]_{i+j-d}$. Note that $\Spec(\CC[\OO])$ is conical with $d=1$ where the grading comes from the dilation action of $\CC^\times$ on $\OO$. 
	\begin{rmk}\label{norm}
		The moment map for a natural $G$-action on $\Spec(\CC[\OO])$ defines a birational map $\Spec(\CC[\OO])\to \overline{\OO}\subset \fg^*$. Therefore, $\Spec(\CC[\OO])$ is a normalization of $\overline{\OO}$.
	\end{rmk}
	
	%
	%
	%
	%
	%
	Let $X$ be an affine conical Poisson variety that has symplectic singularities. Following Namikawa, we can construct a certain affine space $\fP$ (denoted by $\fh$ in \cite{Namikawa}) and a finite group $W$ acting on $\fP$ by reflections. In \cite{Losev4} Losev showed that formal graded quantizations of $X$ are classified by $\fP/W$. 
	Note that the Chevalley-Shepherd-Todd's theorem implies that $\fP/W$ is an affine space. In the next section we describe it explicitly. 
	
	\subsection{Quantizations of $\Spec(\CC[\OO])$}\label{quant affin}
		In this section we follow \cite{Losev4} to compute the space $\fP$ and the group $W$ for $X=\Spec(\CC[\OO])$. We need to recall some constructions first.
		
		Let $\fg$ be a semisimple Lie algebra and $\fl\subset \fg$ be a Levi subalgebra. Let $\fp$ be a parabolic subalgebra containing $\fl$, and $P\subset G$ be the corresponding parabolic subgroup. Suppose we have a nilpotent orbit $\OO'\subset \fl^*$. The coadjoint action of $G$ induces the generalized Springer map $G\times^P (\overline{\OO}'\times \fp^\perp)\to \fg^*$. There is a unique dense orbit $\OO$ in the image. The orbit $\OO$ does not depend on the choice of a parabolic subgroup $P$ and is called induced from $\OO'$ (see \cite{LS}). If the map $G\times^P (\overline{\OO}'\times \fp^\perp)\to \overline{\OO}$ is birational, we say that the orbit $\OO$ is birationally induced from $(\fl, \OO')$ \cite{Losev4}. If an orbit cannot be (birationally) induced from any other proper Levi subalgebra it is called \emph{(birationally) rigid}. We have the following important fact.
			\begin{prop}\cite[Corollary 4.5]{Losev4}
				Every nilpotent orbit can be birationally induced from exactly one up to $G$-conjugation pair $(\fl, \OO')$, such that $\OO'$ is a birationally rigid orbit.
			\end{prop}
		
		Let $L\subset G$ be the Lie subgroup corresponding to the subalgebra $\fl\subset \fg$, and $N_G(L)$ be the normalizer of $L$ in $G$. We set $N_G(L,\OO_0)\subset N_G(L)$ to be the stabilizer of $\OO_0$, and $W({\fl, \OO_0})=N_G(L,\OO_0)/L$.
		\begin{prop}\cite[Proposition 4.6]{Losev4}\label{base}
				$\fP=(\fl/[\fl,\fl])^*$, and $W=W({\fl, \OO_0})$. Therefore, the isomorphism classes of formal graded quantizations of $\Spec(\CC[\OO])$ are classified by the points of an affine space $(\fl/[\fl,\fl])^*/W_{\fl,\OO_0}$. 
		\end{prop}
		Now we can state a refined version of \cref{theorem}.
		\begin{theorem}\label{theorem_ref}
			Let $\fg$ be classical. Then $1$-dimensional $\Gamma$-invariant representations of $\cW$ are classified by points of $(\fl/[\fl,\fl])^*/W({\fl, \OO_0})$.
		\end{theorem}
	
	\subsection{Universal graded Poisson deformation}\label{def sect}
		In this section we discuss the universal graded Poisson deformation of a conical symplectic singularity. Namely, let $X$ be a conical symplectic singularity, and let $\widehat{X}$ be a partial Poisson resolution of $X$, such that $\widehat{X}$ is $\QQ$-factorial and terminal. Existence of such $\widetilde{X}$ is shown in \cite[Proposition 2.1]{LosevSRA} by adapting the argument of \cite{BCHM}. Then we have the following theorem.
		\begin{theorem}\label{universallPoisson}\cite[Proposition 2.6]{Losev4}
			There is a universal graded Poisson deformation $\cX$ of $\widehat{X}$ over $\fP$, such that any graded Poisson deformation $\cX_{S}$ of $\widehat{X}$ over $S$ is obtained via pullback along a map $S\to \fP$.
		\end{theorem}
		For more details on graded Poisson deformations we refer interested reader to \emph{loc.cit.} In this paper we will not use the general theory of graded Poisson deformations. We are mostly interested in the following application of this theorem. Let $\OO$, $\fl$ and $\OO_0$ be as in the previous section. The following proposition recalls the results to be used in \cref{para section}.
		
		\begin{prop}\label{univ def}
			\begin{itemize}
				\item[(a)] $\widehat{X}=G\times^P (\OO_0\times \fp^\perp)$ is a $\QQ$-factorial terminal partial Poisson resolution of $X$.
				\item[(b)] $G\times^P((\fl/[\fl,\fl])^*\times \OO_0\times \fp^\perp)\to (\fl/[\fl,\fl])^*$ is a universal graded Poisson deformation of $\widehat{X}$ over $\fP=(\fl/[\fl,\fl])^*$. 
			\end{itemize}
		\end{prop}
		\begin{proof}
			\begin{itemize}
				\item[(a)] This is \cite[Corollary 4.6]{Losev4}.
				\item[(b)] Part (b) is Proposition 4.7(2) in \cite{Losev4}.
			\end{itemize}
		\end{proof}
	\section{Quantizations of coadjoint orbits}\label{quant orbit}
	\subsection{Hamiltonian $G$-equivariant quantizations}
	In this section we recall constructions and results from \cite{Losev2}. The only new result of this section is \cref{ext}, that is also a straightforward corollary of the construction. From now on all quantizations are assumed to be formal. 
	
	Let $X$ be a smooth symplectic variety over $\CC$. Then both the structure sheaf $\cO_X$ and the sheaf of holomorphic functions $\calH_X$ admit Poisson brackets $\{\bullet, \bullet\}$ induced from the symplectic form. Therefore we may consider (holomorphic) quantizations of the variety $X$. In this section the main example is $X=\OO$, the nilpotent coadjoint orbit with the Kostant-Kirillov symplectic form $\omega_{\chi}(\mu,\eta)=\chi([\mu,\eta])$ for $\mu$, $\eta\in T_{\chi}\OO\subset \fg$. The induced Poisson bracket has degree $-1$. 
	
	Suppose that we have an action of a connected algebraic group $G$ on $X$. We say that a (holomorphic) quantization $\calD$ is $G$-equivariant if the action of $G$ on $\cO_X$ (resp. $\calH_X$) lifts to the action on $\calD$ by algebra automorphisms such that $\hbar$ is $G$-invariant. For $\xi\in \fg$ let $\xi_X$ and $\xi_\calD$ be the derivation of $\cO_X$ (resp. $\calH_X$) and $\calD$ induced by the action of $G$.
	
	Recall that a map $\mu: X\to \fg^*$ is called the moment map if $\{\mu^*(\xi), \bullet\}=\xi_X$ for all $\xi\in \fg$. From now on suppose that $X$ is equipped with a moment map $\mu$. For $X=\OO$ the natural embedding $\OO\hookrightarrow \fg^*$ is a moment map. Let $\calD$ be a $G$-equivariant quantization of $X$. We say that a map $\phi_\hbar: \fg\to \Gamma(X,\calD)$ is a quantum comoment map if $\frac{1}{\hbar}[\phi_\hbar(\xi), \bullet]=\xi_{\calD}$ for all $\xi\in \fg$ and $\phi_\hbar$ coincides with $\mu^*$ modulo $\hbar$. We say that $\calD$ is Hamiltonian if it is equipped with a quantum comoment map $\phi_\hbar$.
	
	In this section we will study $G$-equivariant Hamiltonian (holomorphic) quantizations of $\OO$. Since the action of $G$ is transitive, the idea is to reduce a quantization to its fiber at a point. To do this we need to use the language of quantum jet bundles. 
	\subsection{Quantum jet bundles}\label{quantum jet bundles}
	First, let us give a general idea behind the notion of a quantum jet bundle and how it can be applied to the question of quantizing coadjoint orbits. Any smooth symplectic variety is locally isomorphic to a symplectic vector space, and therefore any function in the formal neighborhood of a smooth point is given by a formal power series. This can be made a global construction, a jet bundle $J^\infty \cO_X$ recalled below. Similarly, any formal quantization of a symplectic variety is locally isomorphic to a Weyl algebra $\CC\langle V, \hbar\rangle / (vw-wv-\hbar^d (v,w))$, and quantum jet bundle $J^\infty \calD$ is a sheaf with fibers isomorphic to the completed Weyl algebra. The idea of \cite{BK} to use quantum jet bundles to study quantizations was adapted to the case of nilpotent coadjoint orbits in \cite{Losev2}, exposition of which we mostly follow (see Section 3.2 of \emph{loc.cit.}). The key idea is to find a natural connection identifying the fibers of quantum jet bundle, which in our situation will come from the transitive action of the group $G$. 
    
    Now we recall the notion of the jet bundle $J^{\infty}\cO_X$. Let $I_\Delta\subset \cO_{X\times X}$ be the ideal of the diagonal $X\subset X\times X$. We set $\cO_{X\times X}^\wedge=\varprojlim_{k\to \infty} \cO_{X\times X}/I_\Delta^k$. Let $\pi_1$, $\pi_2$ be projection maps from $X\times X$ to its factors. We set $J^{\infty}\cO_X=\pi_{1*}(\cO_{X\times X}^\wedge)$. Note that the fiber of $J^\infty\cO_X$ at any point is non-canonically isomorphic to the algebra of formal series in $\dim X$ variables. Let $\nabla$ be a connection on $\pi_{1*}(\cO_{X\times X})$ given locally at point $x$ by $\nabla_\xi(f\otimes g)=\xi(f)\otimes g$ for $\xi\in T_x \OO\subset \fg$. We can uniquely extend it to a continuous in $I_\Delta$-adic topology flat connection on $J^\infty \cO_X$. To simplify notations we will also denote it by $\nabla$. 
	The sheaf of flat sections $(J^\infty \cO_X)^\nabla$ is isomorphic to $\cO_X$.
	
	Recall that $X$ is a symplectic variety. The sheaf $\cO_{X\times X}$ has a natural $\cO_X$-linear (corresponding to the first factor) Poisson bracket that extends to one on $J^\infty\cO_X$. The bracket on the sheaf $\cO_X$ of flat sections coincides with the initial one. The fiber of $J^\infty\cO_X$ at a point $x$ is the algebra $\AA=\CC[[x_1,\ldots, x_n, y_1,\ldots, y_n]]$, where $n=\dim(X)$ with the Poisson bracket given by $\{x_i, x_j\}=0$, $\{y_i, y_j\}=0$ and $\{x_i, y_j\}=\delta_{ij}$. Note that $\AA$ is the completion of the stalk $\cO_x$ with respect to the maximal ideal.
	
	Analogously we can define the holomorphic jet bundle $J^\infty\calH_X$. The fiber of $J^\infty \calH_X$ will be a completion of the stalk $\calH_x$ with respect to the maximal ideal. It is a standard fact that this completion is the same algebra $\AA$, see for example \cite[{Proposition 3}]{GAGA}. 
	
	Now let $\calD$ be a quantization of $X$. We can define $J^\infty \calD$ in the following way. We have the quotient map $\cO_X\otimes \calD\to \cO_X\otimes \cO_X$. Set $\widetilde{I}_\Delta$ to be the inverse image of $I_\Delta$. We define $(\cO_X\otimes \calD)^\wedge=\varprojlim_{k\to \infty}(\cO_X\otimes \calD)/\widetilde{I}_\Delta^k$ and $J^\infty \calD=\pi_{1*}((\cO_X\otimes \calD)^\wedge)$.
	\begin{definition}\label{quantum jet}
		A quantum jet bundle is a triple ($\fD$, $\widetilde{\nabla}$, $\Theta$), where:
		
		$\bullet$ $\fD$ is a pro-coherent sheaf (i.e. limit of coherent sheaves) of $\cO_X[[\hbar]]$-algebras such that $[\fD, \fD]\subset \hbar^2\fD$;
		
		$\bullet$ $\widetilde{\nabla}$ is a flat $\CC[[\hbar]]$-linear connection on $\fD$ such that $\widetilde{\nabla}_\xi$ is a derivation of $\fD$ for all $\xi\in \Vect(X)$;
		
		$\bullet$ $\Theta$ is an isomorphism $\fD/h\fD\to J^\infty \cO_X$ of sheaves of Poisson algebras that intertwines the connections.
	\end{definition}
	A morphism of two quantum jet bundles ($\fD_1$, $\widetilde{\nabla}_1$, $\Theta_1$) and ($\fD_2$, $\widetilde{\nabla}_2$, $\Theta_2$) is a morphism $\Phi: \fD_1\to \fD_2$ of sheaves of $\cO_X[[\hbar]]$-algebras intertwining the connections, and  such that $\Theta_1=\Theta_2\circ \Phi$. By Nakayama's lemma any morphism of two quantum jet bundles is an isomorphism.
	
	One can check that for any quantization $\calD$ the corresponding jet bundle $J^\infty \calD$ is a quantum jet bundle. Moreover, there is the following proposition, where morphisms in the category of quantum jet bundles on $X$ are described above, and morphisms in the category of quantizations of $X$ are described after \cref{conical}.
	
	\begin{prop}\cite[Lemma 3.4]{BK}\label{jet quantization}
		The assignments
		
		$\bullet$ $\calD\to J^\infty\calD$,
		
		$\bullet$ $\fD\to \fD^{\widetilde{\nabla}}$\\
		define mutually quasi-inverse equivalences between the category of quantizations of $X$ and the category of quantum jet bundles on $X$.
	\end{prop}
	
	\subsection{Hamiltonian quantum jet bundles}\label{ham}	
	We are interested in the quantum jet bundles corresponding to Hamiltonian quantizations. We say that a quantum jet bundle ($\fD$, $\widetilde{\nabla}$, $\Theta$) is $G$-equivariant if $\fD$ is equipped with an action of $G$ such that $\widetilde{\nabla}$ and $\hbar$ are $G$-invariant and $\Theta: \fD/h\fD\to J^\infty \cO_X$ is $G$-equivariant. One can check that assignments from \cref{jet quantization} define quasi-inverse equivalences between the category of $G$-equivariant quantizations of $X$ and the category of $G$-equivariant quantum jet bundles on $X$ (see \cite[Corollary 3.3]{Losev2}). Fix the group $G$ to be the simply connected algebraic group corresponding to the simple Lie algebra $\fg$.
	
	Now suppose that $X$ is endowed with a moment map $\mu: X\to \fg^*$. Recall that we have an isomorphism $(J^\infty \cO_X)^\nabla\simeq \cO_X$, and therefore $\Gamma(X, \cO_X)\simeq \Gamma(X, J^\infty \cO_X)^\nabla$. For $\xi\in \fg$ we set $\Phi(\xi)\in \Gamma(X, J^\infty \cO_X)^\nabla$ to be the image of $\mu^*(\xi)\in \Gamma(X, \cO_X)$ under this isomorphism. We have the left, right and diagonal actions of $G$ on $X\times X$. Differentiating, we obtain three maps $\fg\to \Der(\cO_{X\times X})$ given by $\xi\to \xi_{X\times X}^l, \xi_{X\times X}^r, \xi_{X\times X}=\xi_{X\times X}^l+\xi_{X\times X}^r$ respectively. The left action is the connection on $\cO_{X\times X}$, the right one comes from the moment map. Therefore we have $\xi_{J^\infty \cO_X} = \nabla_{\xi_X}+\{\Phi(\xi), \bullet\}$. We say that a $G$-equivariant quantum jet bundle ($\fD$, $\widetilde{\nabla}$, $\Theta$) is Hamiltonian if it is equipped with a map $\Phi_\hbar: \fg\to \Gamma(X, \fD)^{\widetilde{\nabla}}$ such that $\xi_{\fD} = \widetilde{\nabla}_{\xi_X}+\frac{1}{\hbar}[\Phi_\hbar(\xi), \bullet]$ and $\Theta(\Phi_\hbar(\xi))=\Phi(\xi)$. By \cref{jet quantization} they correspond to Hamiltonian quantizations.
	
	\begin{lemma}\cite[Proposition 3.4]{Losev2}\label{quantum comoment}
		Suppose that $H^1_{\DR}(X)={0}$. Then every $G$-equivariant quantization $\calD$ of $X$ admits a unique quantum comoment map $\phi_\hbar: \fg\to \Gamma(X,\calD)$.
	\end{lemma}
	Note that for the homogeneous space $X=\OO$ of the group $G$ we have $H^1_{\DR}(X)=0$. Therefore any $G$-equivariant quantum jet bundle ($\fD$, $\widetilde{\nabla}$, $\Theta$) on $\OO$ admits a unique quantum comoment map $\Phi_\hbar: \fg\to \Gamma(X,\fD)^{\widetilde{\nabla}}$.
	
	We claim that $\Phi_\hbar([\xi, \mu])=\frac{1}{\hbar}[\Phi_\hbar(\xi), \Phi(\mu)]$, so the map $\Phi_\hbar$ extends to an algebra homomorphism $U_\hbar(\fg)\to \Gamma(X,\fD)^{\widetilde{\nabla}}$. Indeed, it is enough to show that for any local section $s$ of $\fD$ we have $[\Phi_\hbar([\xi, \mu]), s]=[[\Phi_\hbar(\xi), \Phi_{\hbar}(\mu)], s]=[\Phi_\hbar(\xi), [\Phi_{\hbar}(\mu), s]]-[\Phi_\hbar(\mu), [\Phi_{\hbar}(\xi), s]]$, where the latter equality is the Jacobi identity. This claim follows by a direct computation from $[\Phi_{\hbar}(\eta), s]=\hbar (\eta_{\fD}s-\widetilde{\nabla}_{\eta_X} s)$ By $\cO_X$-linearity we extend $\Phi_\hbar$ to a homomorphism of sheaves $\cO_X\otimes U_\hbar(\fg)\to \fD$. We have a map $(\cO_X\otimes U_\hbar(\fg))/\hbar (\cO_X\otimes U_\hbar(\fg))\simeq \cO_X\otimes S(\fg)\to \cO_X\otimes \CC[X]$, where the last map is induced by $\mu^*: \fg\to \CC[X]$. Let $\widetilde{I}_{\mu,\Delta}\subset \cO_X\otimes U_\hbar(\fg)$ be the inverse image of $I_\Delta\subset \cO_X\otimes \CC[X]$. We set $J^\infty U_\hbar=\varprojlim_{k\to \infty}\cO_X\otimes U_\hbar(\fg)/\widetilde{I}_{\mu,\Delta}^k$. The sheaf $\cO_X\otimes U_\hbar(\fg)$ admits a flat connection, identifying the flat sections with $U_\hbar(\fg)$. It extends to a natural flat connection $\widetilde{\nabla}_U$ on $J^\infty U_\hbar$. The quantum comoment map $\Phi_\hbar$ extends to a continuous homomorphism $\Psi_\hbar: J^\infty U_\hbar\to \fD$ intertwining the connections $\widetilde{\nabla}_U$ and $\widetilde{\nabla}$. This map plays a crucial role in the classification of $G$-equivariant quantizations of $\OO$.
	
	\subsection{Finite $W$-algebra}\label{W-alg}
	In this section we will give a definition of the $W$-algebra $\cW$ following \cite{Losev3}. 
	
	Recall from \cref{UformS} that $U_\hbar(\fg)$ is a formal quantization of $S(\fg)$, so $U_\hbar(\fg)/(\hbar)\simeq S(\fg)$. Let $m_{\chi}\subset S(\fg)\simeq \CC[\fg^*]$ be the maximal ideal corresponding to the point $\chi$ and $\widetilde{m}_\chi$ be its preimage in $U_\hbar(\fg)$. We have the completion of $U_\hbar(\fg)$ in the point $\chi$ defined by ${U}^{\wedge}_{\hbar}=\varprojlim_{k\to \infty} U_\hbar(\fg)/\widetilde{m}_\chi^k$. Recall that for a finite-dimensional vector space $V$ with a symplectic form $\omega$ we have the homogeneous Weyl algebra $\AA_\hbar(V):=T(V)[[\hbar]]/(u\otimes v-v\otimes u-\hbar^2\omega(u,v))$. Let $\AA_{\hbar}^+(V)=\oplus_{i>0}\AA_{\hbar,i}(V)$ be the augmentation ideal of $\AA_\hbar(V)$. We denote the completion of $\AA_\hbar(V)$ with respect to the ideal $\AA_\hbar^+(V)$ by $\AA^\wedge_{\hbar}(V)$. 
	
	We have a natural action of the centralizer $Z_G(e)\subset G$ on $U_\hbar(\fg)$ that fixes the ideal $\widetilde{m}_\chi$ and therefore extends to ${U}^{\wedge}_{\hbar}$. Set $Q$ to be the reductive part of $Z_G(e)$. We have the following decomposition due to Losev \cite[Theorem 3.3.1]{Losev3}. 
	\begin{prop}\label{decompose}
		There exists a $\CC[[\hbar]]$-algebra $\cW_\hbar$ with an action of $Q\times \CC^\times$ and a $Q\times \CC^\times$-invariant isomorphism $U_\hbar^\wedge\simeq \AA_\hbar^\wedge(V)\widehat{\otimes}_{\CC[[\hbar]]}\cW_\hbar$ for the tangent space $V=T_\chi \OO$ with a symplectic action of $Q$. 
	\end{prop}
	We denote the action of $\CC^\times$ on $\cW_\hbar$ by $\gamma$. Let $\cW_{\hbar, fin}$ be the subalgebra of $\cW_\hbar$ spanned by the elements $w\in \cW_\hbar$ such that $\gamma(t)w=t^iw$ for some $i$. We set $\cW=\cW_{\hbar, fin}/(\hbar-1)$.
	
	For an ideal $J\in \fJ\fd(\cW)$ let $J_\hbar=\bigoplus_{k\ge 0} J\cap \cW_{\le k}\hbar^k$ be the corresponding ideal in the Rees algebra $\cW_{\hbar, fin}$. We set $J^\wedge_{\hbar}\subset \cW_\hbar$ to be the completion of $J_\hbar$.
	
	We have the following important criteria of an ideal in $\cW$ to be of codimension $1$.
	\begin{prop}\cite{Losev2011}\label{U-W}
		For an ideal $J\in \fJ\fd(\cW)$ we set $I^\wedge_\hbar=\AA_\hbar^\wedge(V)\widehat{\otimes}_{\CC[[\hbar]]} J^\wedge_\hbar$. Then $J$ is of codimension $1$ if and only if $(U^\wedge_\hbar/I^\wedge_\hbar)/\hbar (U^\wedge_\hbar/I^\wedge_\hbar)$ coincides with the completion $\CC[\OO]^{\wedge_{\chi}}$ of $\CC[\OO]$ at $\chi$. 
	\end{prop}
	\subsection{Ideals in $W$-algebra vs quantization of orbits}\label{completely prime}
	Let us apply the techniques discussed above to study the set of quantizations of $\OO=G\chi$. Let $H=G_\chi$ be the stabilizer of $\chi$ and $\Gamma=H/H^\circ$ its component group. 
	
	Recall from \cref{ham} that for any graded $G$-equivariant quantization $\calD$ of $\OO$ we have a quantum jet bundle $\fD:=J^\infty \calD$ and the quantum comoment map of quantum jet bundles $\Psi_\hbar: J^\infty U_\hbar\to \fD$.
	\begin{lemma}\cite[Lemma 5.2]{Losev2}\label{Phi}
		The morphism $\Psi_\hbar$ is surjective. 
	\end{lemma}
    As a corollary, the quantum jet bundle $\fD$ is uniquely determined by the corresponding ideal $\Ker(\Psi_\hbar)\subset J^\infty U_\hbar$. We will need a classification of ideals of $J^\infty U_\hbar$. We say that an ideal $I$ of a flat $\CC[[\hbar]]$-algebra $A$ is $\hbar$-saturated if $\hbar a\in I$ implies $a\in I$. Equivalently, $I$ is $\hbar$-saturated if $A/I$ is a flat $\CC[[\hbar]]$-algebra. Note that we have a $\CC^\times$ action on the algebra $U_{\hbar}(\fg)$, and the $\CC^\times$-action on $X=\OO$, coming from the semisimple element in $\fs\fl_2$-triple corresponding to $\chi\in \OO$. Thus, we have the $\CC^\times$ action on the sheaf $\cO_X\otimes U_{\hbar}(\fg)$ that extends to an action on $J^\infty U_\hbar$. We note that if $\calD$ is a $G$-equivariant graded quantization of $\OO$, then $\Ker(\Psi_{\hbar})$ is homogeneous, $G$-stable and $\hbar$-saturated. We have the following description of the set of homogeneous $G$-stable $\hbar$-saturated ideals in $J^{\infty}U_\hbar$. 
	\begin{lemma}\cite[Lemma 5.1]{Losev2}\label{fiber}
		(1) The fiber of $J^\infty U_\hbar$ at $\chi$ is naturally identified with $U_\hbar^\wedge$.
		
		(2) Let $\fJ\fd(J^\infty U_\hbar)$ be the set of homogeneous $G$-stable $\hbar$-saturated ideals in $J^{\infty}U_\hbar$. Taking the fiber of an ideal at $\chi$ defines a bijection between $\fJ\fd(J^\infty U_\hbar)$ and $\fJ\fd(U_\hbar^\wedge)^\Gamma$. The inverse map is given by $I^\wedge_\hbar\to \pi_*(\cO_G\widehat{\otimes}	I^\wedge_\hbar)^H$, where $\pi$ stands for the projection $G\to G/H=X$.
		
		(3) Any element of $\fJ\fd(J^\infty U_\hbar)$ is stable with respect to the connection $\widetilde{\nabla}_U$.
	\end{lemma}	
	%
	The following theorem is the main result of \cite{Losev2}.
	\begin{theorem}\label{quantizations of orbit}\cite[Theorem 1.1]{Losev2}
		There is a well-defined bijective map from the set of $G$-equivariant quantizations of $\OO$ to the set $\fJ\fd^{1}(\cW)^\Gamma$.
	\end{theorem}
	The idea of the proof is to identify both sets with some subset of the set $\fJ\fd(U_\hbar^\wedge)^\Gamma$ of $\hbar$-saturated $\Gamma$-invariant ideals in $U^\wedge_\hbar$. Namely, define $\fJ\fd(U_\hbar^\wedge)^\Gamma_{\OO}\subset \fJ\fd(U_\hbar^\wedge)^\Gamma$ to be the set of all $I^\wedge_{\hbar}\subset \fJ\fd(U_\hbar^\wedge)^\Gamma$, such that $(U^\wedge_{\hbar}/I^\wedge_\hbar)/(\hbar) \simeq \CC[\OO]^{\wedge_{\chi}}$, where the latter stands for the completion of $\CC[\OO]$ in the point $\chi$. First, we note that there is a bijective map from $\fJ\fd^{1}(\cW)^\Gamma$ to $\fJ\fd(U_\hbar^\wedge)^\Gamma_{\OO}$ given by \cref{U-W}. Below we construct a map from the set of $G$-equivariant quantizations of $\OO$ to $\fJ\fd(U_\hbar^\wedge)^\Gamma_{\OO}$, that will be later used in \cref{para section}. For the inverse of this map we refer to \cite{Losev2}. Then the bijective map of \cref{quantizations of orbit} is the composition of these two bijections.
	
	Let $\calD$ be a $G$-equivariant quantization $\calD$ of $\OO$, and let $\fD=\fJ^\infty \calD$ be the corresponding quantum jet bundle. Set $\fJ_{\calD}$ to be the kernel of the map $\Psi_\hbar$, and set $I^\wedge_\hbar\in \fJ\fd(U_\hbar^\wedge)^\Gamma$ to be the fiber of $\fJ_{\calD}$ at $\chi$. By construction, $(U^\wedge_{\hbar}/I^\wedge_\hbar)/(\hbar)\simeq \fJ^\infty \calD_{\chi}/(\hbar)\simeq \fJ^{\infty}\cO_\chi\simeq \CC[\OO]^{\wedge_{\chi}}$.

	Let $I_\hbar$ and $I$ be the corresponding ideals of $U_\hbar(\fg)$ and $U(\fg)$ respectively.
	The algebra $U_\hbar(\fg)/I_\hbar$ is a formal quantization of $\gr [U(\fg)/I]$. The affine Poisson scheme $Y= \Spec(\gr [U(\fg)/I])$ is generically reduced, and its reduced scheme is $\overline{\OO}$. We can microlocalize $U_\hbar(\fg)/I_\hbar$ to a formal quantization of $Y$ in the sense of \cref{microloc}. Restricting to the open subset $\OO\subset Y$ we get a quantization $\calD'$ of $\OO$. We claim the following.
	\begin{lemma}\label{ext}
		The quantizations $\calD'$ and $\calD$ coincide. Therefore $\calD$ extends to a quantization of $Y$.
	\end{lemma}	
    \begin{proof}
        We will show that the corresponding quantum jet bundles $J^\infty\calD$ and $J^\infty\calD'$ coincide. Then the lemma follows from \cref{jet quantization} by considering the sheaves of flat sections. By \cref{Phi}, quantum jet bundles $J^\infty\calD$, $J^\infty\calD'$ are uniquely determined by the kernels $\mathfrak{I}$, $\mathfrak{I}'$ of the quantum comoment maps $\Psi_\hbar: J^\infty U_\hbar\to J^{\infty}\calD$ and $\Psi'_\hbar: J^\infty U_\hbar\to J^{\infty}\calD'$ respectively. 

        By \cref{fiber}(2), $\mathfrak{I}$ and $\mathfrak{I}'$ are uniquely determined by their fibers $\mathfrak{I}_\chi$, $\mathfrak{I}'_\chi$ at the point $\chi$. Since maps $\Psi_\hbar$ and $\Psi'_\hbar$ are $\cO_X$-linear, $\mathfrak{I}_\chi$ and $\mathfrak{I}'_\chi$ are isomorphic to the kernels of the surjective maps on stalks $U_\hbar^{\wedge}\to J^\infty\calD_\chi$ and $U_\hbar^{\wedge}\to J^\infty\calD'_\chi$ respectively. However, the stalks of $J^\infty\calD$ and $J^\infty\calD'$ at the point $\chi\in \OO$ are both equal to $U_\hbar^\wedge/I^\wedge_{\hbar}$. Thus, $\mathfrak{I}_\chi\simeq \mathfrak{I}'_\chi$, and therefore $J^\infty\calD \simeq J^\infty\calD'$.
    \end{proof}
	Let us explain the structure of the proof of \cref{theorem}. The affine Poisson scheme $Y$ may not be normal or even reduced. Therefore we want to lift quantizations to a normal scheme $\Spec(\CC[\OO])$. By \cref{codim 2} it is enough to lift quantizations to symplectic leaves of codimension less then $4$. The smooth parts of both varieties coincide. To lift a quantization in a neighborhood of a point $x$ in a codimension $2$ symplectic leaf, we would like to consider a small open neighborhood of this point which is a product of a Kleinian singularity and an open neighborhood of $x$ in the corresponding codimension $2$ leaf. This is not possible in Zariski topology, which motivates us to work in the complex-analytic topology instead. In \cref{analytification} we microlocalize a quantization of $Y$ to a holomorphic quantization of $Y$. In \cref{lift} we construct a holomorphic quantization of $\Spec(\CC[\OO])$ corresponding to one of $Y$ and consider the corresponding algebraic quantization. In \cref{proof} we show that the main theorem follows from the constructions discussed above.
	
	\section{Analytification of quantization}\label{analytification of quantization}
        Let $I$ be a two-sided ideal of $U(\fg)$. Then $U(\fg)/I$ is a filtered quantization of $\gr [U(\fg)/I]$. Consider the corresponding formal quantization $R^\wedge_\hbar(U(\fg)/I)\simeq U^\wedge_{\hbar}/I^\wedge_{\hbar}$ of $\gr [U(\fg)/I]$ and microlocalize (in the sense of \cref{microloc}) it to a formal quantization $\calD_I$ of $\Spec(\gr [U(\fg)/I])$. The resulting sheaf $\calD_I$ is a quotient of the quantization $\calD$ of $\fg^*$ from \cref{differential operators}. We want to establish a connection between modules over $\calD$ and $\calD_{hol}$ similar to the connection between coherent sheaves and coherent analytic sheaves. First, we need to recall this connection from \cite{GAGA}.   	
 
 \subsection{GAGA}\label{GAGA}
	
	Let $X$ be an algebraic scheme over $\CC$. The set of closed points of $X$ can be endowed with the structure of an analytic space $X_{hol}$. The natural embedding $\imath: X_{hol}\to X$ is continuous, where we consider Zariski topology on $X$ and analytic topology on $X_{hol}$. 
	For a closed point $x\in X$ we have a local ring $\cO_x$ with maximal ideal $m_x$. We define the completion $\widehat{\cO}_x$ as a limit $\varprojlim_{k\to \infty} \cO_x/m_x^k$. Analogously, for a point $x\in X_{hol}$ we consider the analogous completion $\widehat{\calH}_x=\varprojlim_{k\to \infty} \calH_x/m_{x,hol}^k$. Since every regular function in the neighborhood of $x$ gives a germ of holomorphic function at $x$ we have the embedding $\theta_x: \cO_x\to \calH_x$ on stalks that extends to a map $\widehat{\theta}_x: \widehat{\cO}_x\to \widehat{\calH}_x$.
	\begin{prop}\cite[{Proposition 3}]{GAGA} \label{isost}
		The homomorphism $\widehat{\theta}_x: \widehat{\cO}_x\to \widehat{\calH}_x$ is an isomorphism.
	\end{prop}
	\begin{cor}\label{faithfully flat on stalks}
		The algebra $\calH_x$ is faithfully flat over $\cO_x$.
	\end{cor} 
	\begin{proof}
		The algebra $\widehat{\cO}_x\simeq \widehat{\calH}_x$ is faithfully flat over both $\calH_x$ and $\cO_x$. Moreover, we have $\otimes_{\cO_x}\widehat{\cO}_x\simeq \otimes_{\calH_x}\widehat{\calH}_x \circ \otimes_{\cO_x}\calH_x$. The standard argument finishes the proof.
	\end{proof}
	For a coherent algebraic sheaf $\cF$ on $X$ we define a holomorphic sheaf $\cF_{hol}$ on $X_{hol}$ by $\cF_{hol}=\cF\otimes_{\cO_X} \calH_X$. 
	\begin{rmk}\label{inverses}
		In fact, $\cF$ and $\cO_X$ are sheaves on $X$, not $X_{hol}$, so we have to replace them by their inverse images $\imath^{-1}\cF$ and $\imath^{-1}\cO_X$, see e.g. \cite[Section II.5]{Hartshorne1977}, so $\cF_{hol}$ is the pullback of $\cF$ with respect to the morphism of locally ringed spaces $i: (X_{hol}, \calH_X) \to (X, \cO_X)$. In what follows we will omit such details.
	\end{rmk}
	\begin{cor}\label{faithfully flat}
		\begin{itemize}
			\item[(i)] The functor $\bullet_{hol}$ is faithfully flat;
			
			\item[(ii)] The functor $\bullet_{hol}$ sends coherent algebraic sheaves to coherent holomorphic sheaves.
		\end{itemize}
	\end{cor} 
	\begin{proof}
		\begin{itemize}
			\item[(i)] follows from \cref{faithfully flat on stalks} immediately.
			\item[(ii)] follows from $\cO_{X,hol}\simeq \calH_X$.
		\end{itemize}
	\end{proof}
	%
	Now let $Y\subset X$ be a subscheme. Then we have the corresponding ideal sheaf $I_Y\subset \cO_X$. The following is a well-known fact (see for example \cite[Proposition 13.4.4. (iii)]{Taylor}).
	\begin{prop}\label{ideal sheaf}
		The sheaf $I_{Y,hol}\subset \calH_X$ is the ideal sheaf corresponding to the analytic subset $Y$.
	\end{prop}
	\subsection{Analytification of modules over quantization}\label{quantum GAGA}
	We want to establish a connection analogous to \cref{faithfully flat} for modules over quantizations. Recall the sheaves $\calD$ and $\calD_{hol}$ from \cref{differential operators}. Since every algebraic differential operator is a holomorphic one, we have a natural map $\calD_{G,\hbar}\to \calD_{G,hol,\hbar}$ that induces a map $i: \calD\to \calD_{hol}$. Let $x\in \fg^*$ be any point and $\calD_{hol,x}$, $\calD_{x}$ be the corresponding stalks. We have the isomorphisms $\calD_{hol,x}/\hbar{\calD}_{hol,x}\simeq \calH_{x}$ and ${\calD}_{x}/\hbar{\calD}_{x}\simeq \cO_{x}$. Let $\widetilde{m}_{hol,x}$ and $\widetilde{m}_x$ be the inverse images of the maximal ideals, and ${\widehat{\calD}}_{hol,x}$ and ${\widehat{\calD}}_{x}$ be the corresponding completions. The embedding $i$ extends to a map $\widehat{i}: {\widehat{\calD}}_{x}\to {\widehat{\calD}}_{hol,x}$. 
	\begin{lemma}\label{wideiso}
		The map $\widehat{i}$ is an isomorphism.
	\end{lemma}
	\begin{proof}
		First, note that the map ${\widehat{\calD}}_{x}/\hbar{\widehat{\calD}}_{x}\to {\widehat{\calD}}_{hol,x}/\hbar{\widehat{\calD}}_{hol,x}$ induced by $\widehat{i}$ coincides with the map $\widehat{\theta}_x$ from \cref{isost}. The algebra $\calD_{hol,x}$ is $\CC[\hbar]$-flat by construction. \cite[Lemma A.2]{appendix} implies that $\widehat{\calD}_{hol,x}$ is $\CC[[\hbar]]$-flat. So $\widehat{i}$ is a map from the complete and separated $\CC[[\hbar]]$-module ${\widehat{\calD}}_{x}$ to the flat $\CC[[\hbar]]$-module $\widehat{\calD}_{hol,x}$, and $\widehat{i}$ induces an isomorphism on quotients by $(\hbar)$. Nakayama's lemma implies that $\widehat{i}$ is surjective. To show that $\widehat{i}$ is injective, it is enough to verify that the ideal $\Ker \widehat{i}$ is $\hbar$-saturated, i.e. if $\hbar x \in \Ker \widehat{i}$, then $x \in \Ker \widehat{i}$. Then injectivity follows from applying Nakayama's lemma again. The claim about $\hbar$-saturatedness of the ideal is immediate from $\CC[[\hbar]]$-linearity of the map $\widehat{i}$, and the fact that ${\widehat{\calD}}_{hol,x}$ is a free $\CC[[\hbar]]$-module.

	\end{proof} 
	A module $\cM$ over a sheaf of algebras $\cF$ on $X$ is called of finite type if for every point $x\in X$ there is an open neighborhood $U\ni x$ such that $\cM_U$ admits a surjective morphism from a free finitely generated $\cF_U$-module $\cF_U^{\oplus n}$. A module $\cM$ is called coherent if it is of finite type and for any open subset $U$ and any surjective map $\phi: \cF_U^{\oplus n}\to \cM_U$ the kernel of $\phi$ is of finite type. 
	
	We define the functor $\bullet_{hol}$ from the category of right coherent $\calD$-modules to the category of right $\calD_{hol}$-modules by $\cF_{hol}=\cF\otimes_{\calD} \calD_{hol}$ (in the sense of \cref{inverses}). Note that this functor sends $\calD$ to $\calD_{hol}$. 
	
	\begin{prop}\label{quantum faithfully flat}
		The functor $\bullet_{hol}$ is faithfully flat and sends coherent $\calD$-modules to coherent $\calD_{hol}$-modules.
	\end{prop}
	\begin{proof}
		Analogously to \cref{faithfully flat}, it is enough to show that 
		$\calD_{hol,x}$ is faithfully flat over $\calD_x$. By \cref{wideiso}, for any $\calD_x$-module $\cM_x$ we have an isomorphism $\cM_x\otimes_{\calD_x} \widehat{\calD}_x\simeq \cM_{x}\otimes_{\calD_x}\calD_{hol, x}\otimes_{\calD_{hol,x}} \widehat{\calD}_{hol,x}$. Since the completion functor is faithfully flat, both $\bullet\otimes_{\calD_x} \widehat{\calD}_x$ and $\bullet\otimes_{\calD_{hol,x}} \widehat{\calD}_{hol,x}$ are faithfully flat. The proposition follows.
	\end{proof}

	Let $\calD\modd$ and $\calD_{hol}\modd$ be the categories of coherent $\calD$ and $\calD_{hol}$-modules correspondingly. Taking the quotient by the ideal $(\hbar)$ gives functors $\calD\modd\to \Coh(X)$ and $\calD_{hol}\modd\to \Coh_{hol}(X_{hol})$ to the categories of coherent and coherent holomorphic sheaves respectively.
	\begin{prop}\label{commutative}
		The following diagram is commutative.
		\begin{equation*}\label{comm}
			\xymatrix{
				&\calD\modd \ar[r]^{\bullet_{hol}} \ar[d]^{/(\hbar)}& \calD_{hol}\modd \ar[d]^{/(\hbar)}& \\
				&\Coh(X) \ar[r]^{\bullet_{hol}} & \Coh_{hol}(X_{hol})&}
		\end{equation*}
	\end{prop}
	\begin{proof}
		Let $\cM$ be a coherent $\calD$-module. 
		
		Since $\calH_X\simeq \calD_{hol}/\hbar\calD_{hol}$, we have $\cM_{hol}/\hbar \cM_{hol}=\cM_{hol}\otimes_{\calD_{hol}}\calH_{X}=(\cM\otimes_{\calD}\calD_{hol})\otimes_{\calD_{hol}}\calH_{X}=\cM\otimes_{\calD}\calH_X$. Analogously, $(\cM/\hbar \cM)_{hol}=(\cM\otimes_{\calD}\cO_X)\otimes_{\cO_X}\calH_X=\cM\otimes_{\calD}\calH_X$. 
	\end{proof}

	\subsection{Analytification of $\cJ$}\label{analytic ideal}
	Let $\calD'_{\OO}$ be a $G$-equivariant Hamiltonian quantization of $\OO$ and $J$ be the corresponding ideal of $U(\fg)$, see \cref{completely prime}. Let $J_\hbar\subset U_\hbar(\fg)$ be the corresponding ideal. The algebra $U_\hbar(\fg)/J_\hbar$ is a quantization of $\gr[U(\fg)/J]$, and let $\calD'$ be the corresponding quantization of an affine scheme $X= \Spec(\gr [U(\fg)/J])$. That is the extension of $\calD'_\OO$ to $X$ given by \cref{ext}.
	 Set $\cJ\subset \calD$ to be the two-sided ideal generated by $J_\hbar$. Since $J_\hbar$ is a finitely presented $U_{\hbar}$-module, we have $\cJ\in \calD\modd$. 
	
	Note that $\Gamma(\fg^*, \calD/\cJ)=U_\hbar(\fg)/J_\hbar$, so restricting to the affine scheme $X$ the quantization $\calD/\cJ$ is a microlocalization of $U_\hbar(\fg)/J_\hbar$, i.e. $\calD'$. The discussion above implies the following proposition. 
	\begin{prop}\label{quantization-ideal}
		Let $i$ be the natural embedding of $X$ in $\fg^*$. Then we have the following isomorphism $\calD/\cJ\simeq i_*(\calD')$.
	\end{prop}

	The main theorem of this section is as follows.
	\begin{theorem}\label{analytification}
		The sheaf $\cJ_{hol}$ is an $\hbar$-saturated {\bfseries two-sided} ideal in $\calD_{hol}$. Moreover, the quotient $\calD_{hol}/\cJ_{hol}$ is a holomorphic quantization of $X_{hol}$ as an analytic subspace of $\fg^*$. 
	\end{theorem}
	Let us check the first part of the theorem. We have the following more specific statement.
	\begin{prop}\label{two-sided}
		The sheaf $\cJ_{hol}$ is a two-sided ideal in $\calD_{hol}$ generated by $\cJ_\hbar$. 
	\end{prop}
	\begin{proof}
		From the construction of the functor $\bullet_{hol}$ the sheaf $\cJ_{hol}$ is a right ideal in $\calD_{hol}$ generated by $\cJ_\hbar$. 
		We have to check that for every point $x\in \fg^*$ the stalk $\cJ_{hol,x}$ is a two-sided ideal.
		
		By \cref{wideiso}, $\widehat{\cJ}_x\simeq \widehat{\cJ}_{hol,x}$. We set $\cI_x=\widehat{\cJ}_x\cap \calD_{hol, x}$. Note that $\cJ_{hol,x}\subset \cI_x$ and $\cI_x$ is a two-sided ideal. From the universal property of completions the embedding map $i: \cI_x\to \widehat{\cJ}_x$ factors through $f: \widehat{\cI}_x\to \widehat{\cJ}_x$. Note that $f$ is injective. Therefore $\widehat{\cJ}_{hol,x}\subset\widehat{\cI}_x\subset \widehat{\cJ}_x$, so $\widehat{\cI}_x=\widehat{\cJ}_x=\widehat{\cJ}_{hol,x}$. 
		
		Let $g: \cJ_{hol, x}\to \cI_x$ be the natural embedding and set $C=\Coker(g)$. The completion functor $\widehat{\bullet}=\bullet \otimes_{\calD_{hol, x}} \widehat{\calD}_{hol,x}$ is exact on finitely generated $\calD_{hol, x}$-modules, so $\widehat{C}=\Coker(\widehat{g})=0$. From the faithfulness of the completion functor $C=0$, so $\cI_x=\cJ_{hol,x}$, and the latter one is $2$-sided.  		
	\end{proof}
	\begin{proofanal}
		We have to check that $\cJ_{hol, x}$ is $\hbar$-saturated for every $x\in \fg^*$. Recall that $\cJ$ and therefore $\widehat{\cJ}_x$ are $\hbar$-saturated. Hence the intersection $\cI_x=\widehat{\cJ}_x\cap \calD_{hol, x}$ is $\hbar$-saturated. Indeed, if $\hbar f\in \widehat{\cJ}_x\cap \calD_{hol, x}$ for some $f\in \widehat{\calD}_x$ then $f\in \widehat{\cJ}_x$. By \cite[2.4.C]{Matsumura} $\hbar \widehat{\calD}_x\cap \calD_{hol, x}=\hbar\calD_{hol, x}$, so $f\in \calD_{hol, x}$. But we proved that $\cJ_{hol,x}=\cI_x$, so $\cJ_{hol, x}$ is $\hbar$-saturated. Let $I_X\subset \cO_{\fg^*}$ be the ideal sheaf corresponding to the closed subscheme $X\subset \fg^*$. Since $\calD'_{hol}=\calD_{hol}/\cJ_{hol}$ is $\CC[[\hbar]]$-flat, the following diagram is commutative.
		\begin{equation*}
			\xymatrix{
				&& 0 \ar[d]& 0 \ar[d]& 0\ar[d]& &\\
				&0\ar[r]&\hbar\cJ_{hol}\ar[r] \ar[d]& \hbar\calD_{hol}\ar[r] \ar[d]& \hbar\calD'_{hol}\ar[r]\ar[d]& 0&\\
				&0\ar[r]& \cJ_{hol}\ar[r] \ar[d]& \calD_{hol}\ar[r] \ar[d]& \calD'_{hol}\ar[r]\ar[d]& 0&\\
				&0\ar[r]& I_{hol,X}\ar[r] \ar[d]& \calH\ar[r] \ar[d]& i_*(\calH_{X})\ar[r]\ar[d]& 0&\\
				&& 0 & 0 & 0& &}		
		\end{equation*}
		Since the first two columns are exact, the snake lemma implies that the last column is exact. Therefore $\calD'_{hol}$ restricted to $X$ is a holomorphic quantization of $X$. \cref{analytification} is proved. 
		
	\end{proofanal}
	\section{Lift to the normalization}\label{lift}
	\subsection{Singularities in codimension $2$}\label{extend}
	Let $\fg$ be a simple classical Lie algebra, $G$ be the corresponding simply connected algebraic group. Let $\OO\subset \fg^*$ be a nilpotent orbit, $\calD_{\OO}$ be a $G$-equivariant Hamiltonian quantization of $\OO$, and $J\subset U(\fg)$, $X=\Spec(\gr [U(\fg)/J])$ be as in \cref{analytic ideal}. Since $X_{red}\simeq \overline{\OO}$. \cref{norm} implies that that the normalization of $X$ is $\Spec(\CC[\OO])$. In this section we show that any holomorphic quantization $\calD_{hol}$ of $X)_{hol}$ can be lifted to a holomorphic quantization $\widetilde{\calD}_{hol}$ of the normalization $\Spec(\CC[\OO])$. We set $\pi: \Spec(\CC[\OO])\to X$ to be the normalization map. We will denote $\Spec(\CC[\OO])$ by ${\widetilde{X}}$ and omit indexes ${hol}$. To avoid confusion we will denote sheaves of holomorphic functions on $X$ and ${\widetilde{X}}$ by $\calH$ and $\widetilde{\calH}$ correspondingly. The following fact is well-known for experts, but we were unable to find any reference, so will provide a proof in the Appendix A.
	\begin{prop}\label{codim 2}
		Let $X$ be a normal Cohen-Macaulay affine Poisson variety with finitely many symplectic leaves. Let $X_2$ be the union of the smooth part $X^{reg}$ and all symplectic leaves of codimension $2$, and let $i: X_2\to X$ be the natural embedding. Then any (holomorphic) quantization $\calD$ of $X_2$ uniquely extends to a (holomorphic) quantization $\calD_X$ of $X$, and $\calD_X\simeq i_*\calD$. 
	\end{prop}
	
	Therefore it is enough to lift a holomorphic quantization of $X$ to one of ${\widetilde{X}}_2$. The smooth locus of $X$ is the orbit $\OO$, so it is enough to construct a sheaf of $\CC[[\hbar]]$-algebras $\widetilde{\calD}$ on ${\widetilde{X}}_2$ that extends the restriction $\calD_\OO$. Let $i: \OO\to {\widetilde{X}}$ be the natural embedding. We set $\widetilde{\calD}=i_*\calD_{\OO}$. 
	The following is the main result of this section.
	\begin{prop}\label{holquant}
		$\widetilde{\calD}$ is a holomorphic quantization of ${\widetilde{X}}$.
	\end{prop}
	The rest of this section is devoted to the proof of the proposition. First, we need a description of the codimension $2$ symplectic leaves of ${\widetilde{X}}$, and standard neighborhoods of a point in a codimension $2$ leaf.
	\begin{prop}\cite[Theorem 2]{Kraft-Procesi}
		Let $\OO$ be a nilpotent coadjoint orbit in a simple classical Lie algebra and $\OO'$ be an open orbit in the boundary $\delta \OO=\overline{\OO}-\OO$. If $\OO'$ is of codimension $2$ then the singularity of $\overline{\OO}$ in $\OO'$ is smoothly equivalent to an isolated surface singularity of type $A_k$, $D_k$ or $A_k\cup A_k$, where the last one is the non-normal union of two singularities of type $A_k$ meeting transversally in the singular point.
	\end{prop}
	By two singularities meeting transversally in the point we mean the following. The Kleinian singularity of type $A_k$ is an affine scheme $\Spec(A)$, where $A=\CC[x,y,z]/(xy-z^{k+1})$. We have a maximal ideal $(x,y,z)$, corresponding to the singular point $0\in \Spec(A)$, and a quotient map $f: A\to \CC$. Consider a map $A\oplus A\to \CC$ given by $x\oplus y \mapsto f(x)-f(y)$, and set $B$ to be its kernel. We define the singularity $A_k\cup A_k$ to be $\Spec(B)$.
	
	Let $\OO'$ be a codimension $2$ orbit in $\overline{\OO}$, and $S$ be the corresponding singularity. Consider a point $x\in \OO'\subset X$. We say that the open neighborhood $V\subset X$ of $x$ is {\itshape standard} if $V^{reg}=U\times S\setminus \{0\}$ for some open $U\subset \OO'$. Note that standard neighborhoods of $x$ form a base of open subsets of $X$ containing $x$.
	
	If $S$ is of type $A_k$ or $D_k$, the corresponding singularity in the normalization ${\widetilde{X}}$ is equivalent to the one in $\overline{\OO}$. Consider a point $x\in \OO'\subset X$, and set $y=\pi^{-1}(x)$. We have a standard open neigborhood $V\subset X$ of $x$ and set $\widetilde{V}=\pi^{-1}(V)$ to be an open neighborhood of $y$. Then $\widetilde{V}=U\times S\subset {\widetilde{X}}$. 
	
	Assume that the singularity $S$ of the codimension $2$ orbit $\OO'\subset\overline{\OO}$ is of type $A_k\cup A_k$. Let $\OO'$ be the corresponding orbit of codimension $2$. For a point $x\in \OO'\subset X$ we have a standard open neighborhood $V\subset X$. Then $V^{reg}=V_1\sqcup V_2$, where $V_i=U\times (S_i\setminus\{0\})\subset V^{reg}$, where $U\subset \OO'$ is open, and $S_i$ is the singularity of type $A_k$. Let $x_1$, $x_2\in {\widetilde{X}}$ be the two preimages of $x$. We have $\pi^{-1}(V)=\widetilde{V}_1\sqcup \widetilde{V}_2$, where $\widetilde{V}_i=U\times S_i$ is an open neighborhood of $x_i$. Similarly, if the singularity $S$ is of type $A_k$ or $D_k$, we have $V^{reg}=V_1$, and  $\widetilde{V}={V}=U\times S$.

    This geometric description of neighborhoods of codimension $2$ points is sufficient for the main result of this section.
    
    \begin{proofextend}
        Let $x\in X$ be a point of a codimension $2$ orbit, and $y$ be a preimage of $x$ under the map $\widetilde{X}\to X$. By \cref{codim 2} and the discussion above, it is enough to show that $\widetilde{\calD}_{y}/\hbar\widetilde{\calD}_{y}\simeq \widetilde{\calH}_{y}$. Below we will look at the case when the singularity $S$ is of type $A_k\cup A_k$, the proofs for other cases can be obtained in the same way. Let $x_i$, $V_i$, $\widetilde{V}_i$ be as above.
    
        By construction, $\widetilde{\calD}(\widetilde{V}_i)=\calD(V_i)$. We will show that $\widetilde{\calD}_{x_i}/\hbar\widetilde{\calD}_{x_i}\simeq \widetilde{\calH}_{x_i}$. There is a natural map $\widetilde{\calD}(\widetilde{V_1}\sqcup \widetilde{V_2})\simeq \calD(V_1\sqcup V_2)\to \calH(V_1\sqcup V_2)\simeq \widetilde{\calH}(\pi^{-1}(V_1\sqcup V_2))$ with the kernel $\hbar\widetilde{\calD}(\widetilde{V_1}\sqcup \widetilde{V_2})$. Note that by the Hartogs extension theorem $\widetilde{\calH}(\pi^{-1}(V_1\sqcup V_2))\simeq \widetilde{\calH}(\widetilde{V_1}\sqcup \widetilde{V_2})$. Standard open neighborhoods $V^i$ of $x$ form a base of open subsets containing the point $x$. Therefore we can define stalks of sheaves at points $x$, $x_1$, $x_2$ as limits of open neighborhoods $V^i$, $\widetilde{V}^i_1$, $\widetilde{V}^i_2$ correspondingly. We have an induced map on stalks $\widetilde{\calD}_{x_1}\oplus\widetilde{\calD}_{x_2}\to \widetilde{\calH}_{x_1}\oplus \widetilde{\calH}_{x_2}$. It is enough to show that this map is surjective.  
	
	For every $V^i$ there is a natural map $\calD(V^i)\to \calD(V^i_1\sqcup V^i_2)\simeq \widetilde{\calD}(\pi^{-1}(V^i)) $. Taking the limit as $i$ goes to $\infty$ we get a map $\calD_x\to \widetilde{\calD}_{x_1}\oplus \widetilde{\calD}_{x_2}$. Analogously, we have maps $\calH(V^i)\to \calH(V^i_1\sqcup V^i_2)$ and an isomorphism $\calH(V^i_1\sqcup V^i_2)\simeq \calH(\tilde{V^i_1}\sqcup \tilde{V^i_2})$ from the Hartogs extension theorem. Taking the limit we get a map $\calH_x\to \widetilde{\calH}_{x_1}\oplus \widetilde{\calH}_{x_2}$, that fits in the following commutative diagram 
	\begin{equation*}\label{normalizationdiag}
		\xymatrix{
			&\calD_x \ar[r] \ar[d]& \widetilde{\calD}_{x_1}\oplus \widetilde{\calD}_{x_2} \ar[d]& & &\\
			&\calH_x \ar[r] & \widetilde{\calH}_{x_1}\oplus \widetilde{\calH}_{x_2} \ar[r] & \calH_{x, {\OO'}}\ar[r]& 0&}\tag{1}
	\end{equation*}
	Here $\calH_{x,{\OO'}}$ stands for the ring of germs of analytic functions on ${\OO'}$ at the point $x$. Note that the lower sequence is exact at the middle. Since $\calD$ is a holomorphic quantization of $X$ the left vertical map is surjective. Therefore it is enough to show that the composition $\phi: \widetilde{\calD}_{x_1}\oplus \widetilde{\calD}_{x_2}\to \widetilde{\calH}_{x_1}\oplus \widetilde{\calH}_{x_2}\to \calH_{x, {\OO'}}$ is surjective. Let us choose an element $f\in \calH_{x,{\OO'}}$. 
	We have a surjective map $\calH_{x, X}\to \calH_{x,{\OO'}}$, and set $g$ to be a preimage of $f$. Consider a standard open neighborhood $V\subset X$ of $x$, such that there is a function $g\in \calD(V)/\hbar\calD(V)\subset \calH(V)$ giving the germ $g\in \calH_{x, X}$. We can choose a lift $\widetilde{g}\in \calD(V)$. Since $V_1\subset V$ is an open subset, we have the restriction map $r: \calD(V)\to\calD(V_1)\simeq \widetilde{\calD}(\widetilde{V_1})$, and set $\widetilde{f}=r(\widetilde{g})$. Consider the germ of $\widetilde{\calD}_{x_1}$ that is represented by $\widetilde{f}$. Then $\phi(\widetilde{f})=f$. So the right vertical arrow in Diagram (\ref{normalizationdiag}) is surjective. 
 
    To finish the proof of \cref{holquant} let $\rho: \widetilde{X}\to \overline{\OO}$ be the natural map, and let $j: \OO\to \widetilde{X}_2$, and $i^2: \widetilde{X}_2\to \widetilde{X}$ be the natural embeddings. For a point $x\in \widetilde{X}$ let $\OO_x\subset \overline{\OO}$ be the unique nilpotent orbit containing $\rho(x)$. By construction of $\widetilde{\calD}$, the map $\widetilde{\calD}_x/\hbar\widetilde{\calD}_x\to \calH_x$ is an isomorphism over all points $x\in \OO$. The argument above shows that the map is an isomorphism if $\OO_x$ is of codimension $2$ in $\overline{\OO}$. The map $\rho$ is finite, and therefore for every $x\in \widetilde{X}_2$ either $\OO_x\simeq \OO$, or $\OO_x$ is of codimension $2$ in $\overline{\OO}$. Therefore, $j_*\calD_{\OO}$ is a holomorphic quantization of $\widetilde{X}_2$. By \cref{codim 2}, $i^{2}_* j_*\calD_{\OO}=i_*\calD_{\OO}$ is a holomorphic quantization of $\widetilde{X}$. That completes the proof.
    \end{proofextend}
    
    We remark that a quantization $\calD$ of $\widetilde{X}$ is uniquely determined by its restriction to $\OO$. Indeed, let $i: \OO\to \widetilde{X}$ denote the embedding. By the adjunction, there is a natural map of $\CC[[\hbar]]$-flat sheaves of algebras $\calD\to i_*\calD_{\OO}$. Since both sheaves are holomorphic quantizations of $\widetilde{X}$, this map is an isomorphism. 
	\subsection{From an analytic quantization to an algebraic one}
	Let $\calD$ be an analytic quantization of $\Spec(\CC[\OO])$ and $\cA=\Gamma(\Spec(\CC[\OO]), \calD)$ be the algebra of global sections. We have an action of $\CC^\times$ on $\cA$ and set $\cA_{alg}$ to be the $\hbar$-adic completion of $\cA_{fin}$.  
	
	\begin{prop}\label{alg}
		$\cA_{alg}$ is an algebraic formal quantization of $\CC[\OO]$.
	\end{prop} 
	\begin{proof}
		$\cA_{alg}$ is a complete and separated $\CC[[\hbar]]$-algebra with a $\CC^\times$-action by algebra automorphisms induced from the $\CC^\times$-action on $\cA$. This action is rational on all quotient $\cA_{alg}/(\hbar^k)$, and $[\cA_{alg}, \cA_{alg}]\subset \hbar\cA_{alg}$. Therefore it is enough to show that $\cA_{alg}/\hbar\cA_{alg}\simeq \CC[\OO]$. 
		
		$\Spec(\CC[\OO])$ is an affine scheme and therefore a Stein analytic space. By \cite[Theorem B]{Cartan1953}, $H^1(\Spec(\CC[\OO]), \calH)=0$. Therefore $\cA=\Gamma(\Spec(\CC[\OO]), \calD)$ is a formal quantization of $\Gamma(\Spec(\CC[\OO]), \calH)$. Note that $\cA_{alg}/\hbar\cA_{alg}\simeq \cA_{fin}/\hbar\cA_{fin}$ and $\hbar\cA_{fin}=\cA_{fin}\cap \hbar\cA$. Therefore $\cA_{fin}/\hbar\cA_{fin}\simeq (\cA/h\cA)_{fin}=\Gamma(\Spec(\CC[\OO]),\calH)_{fin}$. For any regular function $f\in \CC[\OO]$ we have $f\in \Gamma(\Spec(\CC[\OO]),\calH)_{fin}$. In the opposite direction, for any $f\in \Gamma(\Spec(\CC[\OO]),\calH)$ a germ $\tilde{f}$ of the analytic function $f$ at $0$ can be written as a formal power series converging in a neighborhood of $0$. Moreover, we have a $\CC^\times$-action on the algebra of formal power series, and for a $\CC^\times$-finite function $f$ we can choose the lift $\tilde{f}$ to be $\CC^\times$-finite. Since the action of $\CC^\times$ on $\Spec(\CC[\OO])$ is conical and has only positive weights, a $\CC^\times$-finite element in the algebra of power series is a polynomial, and thus $f$ is a regular function. Therefore $\Gamma(\Spec(\CC[\OO]),\calH)_{fin}=\CC[\OO]$ and $\cA_{alg}/\hbar\cA_{alg}\simeq \CC[\OO]$. 
	\end{proof}
	
	We can microlocalize $\cA_{alg}$ to an algebraic formal quantization $\calD_{alg}$ of $\Spec(\CC[\OO])$. Therefore starting from a $G$-equivariant quantization $\calD_{\OO}$ of $\OO$ we can obtain a holomorphic quantization of some non-reduced scheme $X$ with $X_{red}=\overline{\OO}$ by \cref{analytification}, lift it to a holomorphic quantization of ${\widetilde{X}}=\Spec(\CC[\OO])$ by \cref{holquant}, and construct the corresponding algebraic quantization of $\Spec(\CC[\OO])$ by \cref{alg}. Analogous to the holomorphic case, a quantization $\calD_{alg}$ of $\Spec(\CC[\OO])$ is completely determined by its restriction to $\OO$. We proved the following theorem.
	\begin{theorem}\label{extension}
		Assume that $\fg$ is classical simple Lie algebra. Every $G$-equivariant quantization of $\OO$ can be uniquely extended to a $G$-equivariant quantization of $\Spec(\CC[\OO])$.
	\end{theorem}

	\subsection{Proof of \cref{theorem}}\label{proof}
	
	Now we can prove the theorem stated in the introduction. 
	More precisely, we prove the refined statement.
	\begin{theorem}
		Let $\fg$ be classical. Then $1$-dimensional representations of $\cW$ are classified by points of $(\fl/[\fl,\fl])^*/W({\fl, \OO_0})$.
	\end{theorem}
	\begin{proof}
		Suppose that $\widetilde{\calD}$ is a graded formal quantization of $\Spec(\CC[\OO])$. 
		By \cite[Remark 3.24]{Losev4} the action of $G$ on $\CC[\OO]$ uniquely extends to the action of $G$ on $\Gamma(\Spec(\CC[\OO]), \widetilde{\calD})$, such that $\widetilde{\calD}$ become a $G$-equivariant quantization. Restricting to the open subset $\OO$ we get a $G$-equivariant quantization $\calD$ of $\OO$. By \cref{extension} it has a unique $G$-equivariant extension to $\Spec(\CC[\OO])$ that is $\widetilde{\calD}$. Therefore the constructed maps $\calD\to \widetilde{\calD}$, $\widetilde{\calD}\to \calD$ between the sets of $G$-equivariant quantizations of $\OO$ and quantizations of $\Spec(\CC[\OO])$ are inverse to each other. 
		By \cref{base}, $G$-equivariant quantizations of $\Spec(\CC[\OO])$ are classified by the points of $(\fl/[\fl,\fl])^*/W({\fl, \OO_0})$. From \cref{quantizations of orbit} the set of $G$-equivariant quantizations of $\OO$ is in bijection with $\fJ\fd^{1}(\cW)^\Gamma$. The present theorem follows.
	\end{proof}
	
	\section{Parabolic induction for $1$-dimensional representations of $W$-algebras}\label{para section}
		The main goal of this section is to give a geometric analogue of the parabolic induction for $W$-algebras, namely parabolic induction for quantizations of nilpotent orbits. First, we recall the most important properties of the parabolic induction for $W$-algebras defined in \cite{Losev_parabolic}. Then we give a geometric construction of parabolic induction for the quantizations of nilpotent orbits, and show that under the bijection from \cref{quantizations of orbit} the two inductions coincide. Thus, we can view our construction as geometric analogue of a parabolic induction of $\Gamma$-invariant $1$-dimensional representations of $W$-algebras. We use to prove \cref{theorem2} by constructing the universal quantization of $G\times^P (\OO_0\times \fp^\perp)$. 
		
		\subsection{Correspondence between ideals in $\cW$ and $U(\fg)$}\label{daggers}
			First, we need to recall the restriction and induction maps between ideals in $U(\fg)$ and $\cW$.
			In this section we mostly follow \cite{Losev3}. For a more general case of restriction and induction functors we refer to \cite{Losev2011}. 
			
			Let $e\in \fg$ be a nilpotent element, and $\cW$ be the corresponding finite $W$-algebra. Consider an ideal $J\subset \cW$, let $J_\hbar\subset \cW_{\hbar, fin}$ be the corresponding Rees ideal, and set $J^\wedge_{\hbar}\subset \cW_\hbar$ to be its completion. Recall from \cref{decompose} that we have an isomorphism $U^\wedge_{\hbar}\simeq \AA_{\hbar}(V)\widehat{\otimes}_{\CC[[\hbar]]} \cW_\hbar$, and let $I^\wedge_{\hbar}:=\AA_{\hbar}(V)\widehat{\otimes}_{\CC[[\hbar]]} J^\wedge_\hbar\subset U^\wedge_{\hbar}$. Let $I_\hbar=I^\wedge_{\hbar}\cap U_\hbar(\fg)\subset U_\hbar(\fg)$, and $I=I_\hbar/(\hbar-1)\subset U(\fg)$. We denote the ideal $I$ by $J^\dagger$, and say that the ideal $I$ is the {\itshape induction} of the ideal $J$.
			
			In the opposite direction, let $I\subset U(\fg)$ be an ideal. Let $I_\hbar\subset U_\hbar(\fg)$ be the Rees module corresponding to $I$ ,and $I_\hbar^\wedge\subset U^\wedge_{\hbar}$ to be the completion of $I_\hbar$ at the point $e$. Since $U^\wedge_{\hbar}\simeq \AA_{\hbar}(V)\widehat{\otimes}_{\CC[[\hbar]]} \cW_\hbar$, we have an adjoint action of $V$ on $I_\hbar^{\wedge}$. Let $J^\wedge_{\hbar}=(I_\hbar^\wedge)^{\ad V}\subset \cW_\hbar$ be the invariants under this action. Let $J_\hbar=J^\wedge_{\hbar}\cap \cW_{\hbar, fin}$, and $J=J_\hbar/(\hbar-1)\subset \cW$.  We denote the ideal $J$ by $I_{\dagger}$, and say that the ideal $J$ is the {\itshape restriction} of the ideal $I$.
			
		\subsection{Parabolic induction for $W$-algebras}\label{para W}
		 	
		Let $\OO\subset \fg^*$ be a nilpotent orbit birationally induced from a nilpotent orbit $\OO_0\subset \fl^*$ for some Levi subalgebra $\fl\subset \fg$. Let $\cW$ and $\cW_0$ be the corresponding finite $W$-algebras, $U$ and $U_0$ be the universal enveloping algebras for $\fg$ and $\fl$ correspondingly, and $\Gamma$ and $\Gamma_0$ be the component groups of the stabilizers of $\chi\in \OO$ and $\chi_0\in \OO_0$ in $G$ and $L$ correspondingly. We denote the induction and restriction maps for ideals in $\cW_0$ and $U_0$ by $\bullet^{\dagger_0}$ and $\bullet_{\dagger_0}$ correspondingly. Recall from \cite{Dixmier} the notion of parabolic induction from the set of ideals in $U_0$ to the set of ideals in $U$. Similarly, one can define a map between the ideals for completed homogenized universal enveloping algebras. Let $I^\wedge_{0,\hbar}\subset U_{0,\hbar}^{\wedge_{\chi_0}}$ be an ideal, and let $M_0=U_{0,\hbar}^{\wedge_{\chi_0}}/I^\wedge_{0,\hbar}$. Note that the maps $U_\hbar(\fp)\to U_\hbar$ and $U_\hbar(\fp)\to U_{0,\hbar}$ naturally extend to the maps $U_\hbar(\fp)^{\wedge_\chi}\to U^{\wedge_\chi}_\hbar$ and $U_\hbar(\fp)^{\wedge_\chi}\to U^{\wedge_{\chi_0}}_{0,\hbar}$. Let $M=U_\hbar^{\wedge_\chi}\widehat{\otimes}_{U_\hbar(\fp)^{\wedge_\chi}} M_0$, and $I^\wedge_\hbar=\Ann_{U_\hbar^{\wedge_\chi}}(M)$. We say that $I^\wedge_\hbar$ is parabolically induced from $I^\wedge_{0,\hbar}$. To simplify notations, we write $U_{0,\hbar}^{\wedge}$ for $U^{\wedge_{\chi_0}}_{0,\hbar}$ and $U_\hbar^\wedge$ for $U_\hbar^{\wedge_\chi}$ respectively.
		
		In \cite{Losev_parabolic} Losev defined a parabolic induction for finite-dimensional representations of $W$-algebras. Namely, let us denote by $\Rep_{fd}(\cW)$ and $\Rep_{fd}(\cW_0)$ the sets of finite-dimensional representations of $\cW$ and $\cW_0$ correspondingly. Losev constructed a map $\rho: \Rep_{fd}(\cW_0)\to \Rep_{fd}(\cW)$. Let us recall the most important properties of the map $\rho$.
		
		\begin{prop}\label{rho}
			Consider any $M_0\in \Rep_{fd}(\cW_0)$, and set $M=\rho(M_0)$. The following is true:
			\begin{itemize}
				\item[1)] \cite[Theorem 1.2.1]{Losev_parabolic} $\rho$ preserves dimension of the representation, i.e. $\dim \rho(M_0)=\dim M_0$;	
				\item[2)] \cite[Corollary 6.4.2]{Losev_parabolic} Set $I_0=\Ann_{\cW_0}(M_0)$ and $I=\Ann_{\cW}(M)$. Then $J=I^{\dagger}$ is obtained from $J_0=I_0^{\dagger_0}$ by the parabolic induction.
                \item[3)] \cite[Proof of Theorem 6.4.1]{Losev_parabolic} Let $J_0$ and $J$ be as in 2), and $J_0^{\wedge}\subset U_{0,\hbar}^{\wedge}$, $J^\wedge\subset U_\hbar^\wedge$ be the corresponding completed ideals. Then $J^\wedge$ is parabolically induced from $J_0^\wedge$.
			\end{itemize}
		\end{prop}

		\subsection{Quantum Hamiltonian reduction}\label{Ham}
			One can argue similarly to define the parabolic induction for quantizations of nilpotent orbits. For that we need to recall the notion of a quantum Hamiltonian reduction. There will be no new results in this section. We refer to \cite[Section 3]{Losev_isofquant} for proofs of given statements.
   
            Let $X$ be an affine conical Poisson variety with an action of a reductive group $G$ respecting the Poisson bracket and commuting with the action of $\CC^\times$ coming from grading. Let $\fg$ be the Lie algebra of $G$, and $\mu: X\to \fg^*$ be the moment map for the action of $G$. The fiber $\mu^{-1}(0)\subset X$ is a smooth subvariety closed under the action of $G$, see \cite[Section 26]{Guillemin}. Note that $\mu^{-1}(0)$ is affine, and therefore the categorical quotient $\mu^{-1}(0)/G$ is well-defined. We call the variety $\mu^{-1}(0)/G$ the {\itshape Hamiltonian reduction} of $X$ and denote it by $X\ham G$. It is easy to see that $X\ham G$ is an affine conical Poisson variety. 
			 
			The quantum Hamiltonian reduction is a method to get a quantization of $X\ham G$ starting from a quantization of $X$.  Let $\calD$ be a $G$-equivariant formal graded quantization of $X$. We say that a map $\Phi_\hbar: \fg\to \Gamma(X,\calD)$ is a {\itshape quantum comoment map} if $[\Phi_\hbar(\xi), \bullet]=\hbar^2\xi_{\calD}$ for all $\xi\in \fg$, and $\Phi_\hbar$ coincides with the comoment map $\mu^*: \fg\to \CC[X]$ modulo $\hbar$. Suppose that a quantum comoment map exists, and let us fix $\Phi_\hbar$. 

			Let $\cA_\hbar=\Gamma(X,\calD)$, and let $I\subset \cA_\hbar$ be an ideal generated by $\Phi_\hbar(\xi)$ for all $\xi\in \fg$. Let $\cI\subset \calD$ to be the sheaf generated by $I$, then $\calD/\cI$ is a sheaf supported on $\mu^{-1}(0)\subset X$. Let $\pi: \mu^{-1}(0)\to X \ham G$ be the natural projection, and let $\calD\ham G=\pi_*(\calD/\cI)^G$. We call the resulting sheaf $\calD\ham G$ the {\itshape quantum Hamiltonian reduction} of $\calD$. 
		\begin{prop}\label{quantum Ham}\cite[Section 3.4]{Losev_isofquant}
			Assume that the action of $G$ on $\mu^{-1}(0)$ is free. Then $\calD\ham G$ is a formal graded quantization of $X\ham G$.
		\end{prop}

		\subsection{Parabolic induction for quantizations of nilpotent orbits}\label{para quant}
		
		Assume that $\OO$ is birationally induced from $\OO_0$, and let $\calD_{{\OO_0}}$ be an $L$-equivariant quantization of $\OO_0$. By \cref{extension} it can be uniquely extended to a quantization $\calD_0$ of $\Spec(\CC[\OO_0])$. Let $Z=T^*G\times \Spec(\CC[\OO_0])$, and let $\calD_Z=\underline{\calD}_{G,\hbar}\boxtimes \calD_0$ be a quantization of $Z$. 
		
%
		Recall that $Z\simeq G\times \fg^*\times \Spec(\CC[\OO_0])\simeq G\times (\Spec(\CC[\OO_0])\times \fg^*)$. Pick a parabolic subgroup $P\subset G$ with Levi decomposition $P=LN$. Let $P^-$ be the opposite parabolic, and $P^-=LN^-$ be its Levi decomposition. We have a natural action of the parabolic subgroup $P$ on both $G$ and $\Spec(\CC[\OO_0])\times \fg^*$. Consider the diagonal action of $P$ on $Z$. It is Hamiltonian with a moment map $\mu: Z\to \fp^*$, and $\mu^{-1}(0)\simeq G\times (\Spec(\CC[\OO_0])\times \fp^\perp)$. Note that the comoment map $\fp\to \CC[Z]$ lifts to the quantum comoment map $\fp\to \Gamma(Z,\calD_Z)$. Indeed, the quantization $\calD_0$ of $\Spec(\CC[\OO_0])$ has a quantum comoment map by \cref{quantum comoment}, and $\underline{\calD}_{G,\hbar}$ is endowed with a natural quantum comoment map given by $\Phi_\hbar(\xi)=\xi$. Set ${\widetilde{X}}=Z\ham P\simeq G\times^P (\Spec(\CC[\OO_0])\times \fp^\perp)$, and $\calD_{\widetilde{X}}=\calD_Z\ham P$. The action of $P$ on $\mu^{-1}(0)$ is free, so by \cref{quantum Ham} $\calD_{\widetilde{X}}$ is a quantization of ${\widetilde{X}}$. Recall that $\OO$ is birationally induced from $\OO_0$, so $\OO\subset {\widetilde{X}}$ is the open $G$-orbit. Restricting $\calD$ to $\OO$ we get a quantization $\calD_{\OO}$. We say that $\calD_{\OO}$ is {\itshape induced} from $\calD_{\OO_0}$ and denote it by $\Ind(\calD_{\OO_0})$. Note that $\calD_Z$ has a natural left action of $G$ that commutes with a diagonal action of $P$, and $\calD_Z$ is a $G$-equivariant quantization of $Z$. Therefore, $\calD_{\widetilde{X}}$ and $\calD_{\OO}$ are $G$-equivariant quantizations of ${\widetilde{X}}$ and $\OO$ correspondingly. 
		
		Let $Y=T^*N^{-}\times \Spec(\CC[\OO_0])\subset {\widetilde{X}}$, and consider $\chi\in \OO\cap Y\cap \fp^*$. Let $\chi_0$ be the image of $\chi$ under the projection $T^*N^{-}\times \Spec(\CC[\OO_0])\to \Spec(\CC[\OO_0])$. Since $\chi\in \OO$, we have $\chi_0\in \OO_0$.
		\begin{prop}\label{compl ind}
			Let $I^\wedge_{0,\hbar}$ and $I^\wedge_\hbar$ be the ideals in $U^{\wedge}_{0,\hbar}$ and $U_\hbar^{\wedge}$ corresponding in the sense of \cref{completely prime} to the quantizations $\calD_{\OO_0}$ and $\calD_{\OO}=\Ind(\calD_{\OO_0})$ respectively. Then $I^\wedge_\hbar$ is parabolically induced from $I^\wedge_{0,\hbar}$.
		\end{prop}
    \begin{proof}
        First, we show that the restriction of $\calD_{\widetilde{X}}$ to $Y$ is isomorphic to $\calD_{N^{-}, \hbar} \boxtimes_{\CC[[\hbar]]} \calD_{0}$.
        
        Note that $Y$ is the Hamiltonian reduction of a $G$-invariant subset $Z'=T^*(N^{-}P)\times \Spec(\CC[\OO_0])\subset Z$ by the action of $P$. The restriction of $\calD_Z$ to $Z'$ is isomorphic to  $\calD_{Z'}=\underline{\calD}_{N^{-}P,\hbar}\boxtimes_{\CC[[\hbar]]} \calD_{0}$. Therefore, $\calD_{\widetilde{X}}|_Y\simeq \calD_{Z'}\ham G\simeq \calD_\hbar (N^{-}) \boxtimes_{\CC[[\hbar]]} \calD_{0}$, as claimed above. 
        
        Now we use it to prove the proposition. Consider the quantum comoment maps $U_\hbar^\wedge\to (\fJ^\infty \calD_{\OO})_{\chi}$ and $U_{0, \hbar}^\wedge\to (\fJ^\infty \calD_{\OO_0})_{\chi_0}$. By the discussion above, $(\fJ^\infty \calD_{\OO})_{\chi}\simeq \AA_\hbar(\fn^-)\widehat{\otimes}_{\CC[[\hbar]]} (\fJ^\infty \calD_{\OO_0})_{\chi_0}$. Arguing as \cite[Proof of Theorem 6.4.1]{Losev_parabolic}, we have a natural map $U_\hbar^{\wedge}\to \AA_\hbar(\fn^-)\widehat{\otimes}_{\CC[[\hbar]]} U_{0, \hbar}^\wedge$, and the following diagram commutes.

        \begin{equation*}\label{normalization}
		\xymatrix{
			& U_\hbar^{\wedge} \ar[d] \ar[r]& \AA_\hbar(\fn^-)\widehat{\otimes}_{\CC[[\hbar]]} U_{0, \hbar}^\wedge \ar[d] &\\
			&(\fJ^\infty \calD_{\OO})_{\chi}\ar[r] & \AA_\hbar(\fn^-)\widehat{\otimes}_{\CC[[\hbar]]} (\fJ^\infty \calD_{\OO_0})_{\chi_0}&}
	\end{equation*}

        The remainder of the proof follows the lines of \cite[Proof of Lemma 6.4.3]{Losev_parabolic}
    \end{proof}
  
		
		\begin{prop}\label{ind same}
			Let $M_0$ be a $1$-dimensional $\Gamma_0$-invariant representation of $\cW_0$, and $\calD_{\OO_0}$ be the corresponding $L$-equivariant quantization of $\OO_0$. Then $M=\rho(M_0)$ is a $\Gamma$-invariant $1$-dimensional representation of $\cW$, and the corresponding $G$-equivariant quantization of $\OO$ is isomorphic to $\calD_{\OO}=\Ind(\calD_{\OO_0})$.
		\end{prop}
		\begin{proof}
			Recall from \cref{completely prime} that the bijection between the set of $1$-dimensional representations of $\cW$ and the set of $G$-equivariant quantizations of $\OO$ goes through the identification of both sides with the subset $\fJ\fd(U_\hbar^\wedge)^\Gamma_{\OO}$ of the set of $\Gamma$-invariant $\hbar$-saturated ideals of $U^\wedge_{\hbar}$. 
			
			Set $J\subset \cW$ and $J_0\subset \cW_0$ to be the annihilators of $M$ and $M_0$ correspondingly, and let $I^\wedge_{M, \hbar}\subset U_\hbar^\wedge$ and $I^\wedge_{M_0, \hbar}\subset U_{0,\hbar}^\wedge$ be the ideals corresponding to $M$ and $M_0$ via \cref{U-W}. By 3) of \cref{rho}, $I^\wedge_{M, \hbar}$ is parabolically induced from $I^\wedge_{M_0, \hbar}$. Together with \cref{compl ind} that implies the proposition.

		\end{proof}
		\subsection{Proof of \cref{theorem2}}\label{proof_2}
		In this section we prove \cref{theorem2}. Let $(\fl, \OO_0)$ be the pair of a Levi subalgebra $\fl\subset \fg$, and a nilpotent birationally rigid $\OO_0\subset \fl^*$, such that $\OO$ is birationally induced from $\OO_0$. Recall that by \cref{base} the Namikawa space $\fP$ of $X$ is identified with $(\fl/[\fl,\fl])^*$. Let $\widehat{X}=G\times^P (\OO_0\times \fp^\perp)$, and $\widehat{X}_{\fP}=G\times^P ((\fl/[\fl,\fl])^*\times \OO_0\times \fp^\perp)$. Note that we have a natural projection map $\phi: \widehat{X}_{\fP}\to \fP$, and $\phi^{-1}(0)=\widehat{X}$.
		
		By \cref{univ def}, $\widehat{X}_{\fP}$ is the universal deformation of $\widehat{X}$. Therefore, by \cite[Section 6.2]{BK} if we have a quantization $\calD_{\fP}$ of $\widehat{X}_{\fP}$, then all quantizations of $\widehat{X}$ are obtained in a unique way via pullback of $\calD_{\fP}$. 
		
		By \cref{base}, there is a unique quantization $\calD_{0}$ of $\Spec(\CC[\OO_0])$. Let $\calD_{0,\fl}=\cO_{(\fl/[\fl,\fl])^*}[[\hbar]]\widehat{\otimes}_{\CC[[\hbar]]} \calD_0$ be the $L$-equivariant quantization of $\Spec(\CC[\OO_0])\times (\fl/[\fl,\fl])^*$. We have a natural map $(\fl/[\fl,\fl])^*\times \Spec(\CC[\OO_0])\to (\fl/[\fl,\fl])^*$. For $\lambda\in (\fl/[\fl,\fl])^*$ the quantization $\calD_{\cO_0,\lambda}$ of $\OO_0$ extends to a $G$-equivariant quantization $\calD_{0,\lambda}$ of $\Spec(\CC[\OO_0])$ that is isomorphic to the pullback of $\calD_{0,\fl}$ along the embedding of a point $\lambda$. Let $\calD_{Z,\fP}=\calD_{G,\hbar}\widehat{\otimes}_{\CC[[\hbar]]} \calD_{\fl}$ be a quantization of $T^*G\times \Spec(\CC[\OO_0])\times (\fl/[\fl,\fl])^*$, and $\calD_{\widetilde{X},\fP}=\calD_{Z,\fP}\ham P$ be a quantization of $\widetilde{X}_{\fP}:=(T^*G\times (\fl/[\fl,\fl])^*\times \Spec(\CC[\OO_0]))\ham P$. Restricting to the open subset $\widehat{X}_{\fP}\subset \widetilde{X}_{\fP}$ we get a quantization $\calD_{\fP}$ of $\widehat{X}\times (\fl/[\fl,\fl])^*$. By the discussion above, all quantizations of $\widehat{X}$ are obtained via pullback of $\calD_{\fP}$. For any $\lambda\in (\fl/[\fl,\fl])^*$, let $\calD_\lambda$ be the pullback of $\calD_{\fP}$ under the embedding of point $\lambda$. By construction, $\calD_\lambda$ is parabolically induced from $\calD_{0,\lambda}$. Therefore, any $G$-equivariant quantization of $\widehat{X}$ and thus $\OO$ is induced from an $L$-equivariant quantization of $\OO_0$. We are ready to deduce \cref{theorem2}.
		
%
%
%
%
		\begin{proof2}
			Let $M$ be a $1$-dimensional $\Gamma$-invazriant representation of $\cW$, and $\calD$ be the corresponding $G$-equivariant quantization of $\OO$. By the argument above, there is an $L$-equivariant quantization $\calD_0$ of $\OO_0$, such that $\calD\simeq \Ind(\calD_0)$. Let $M_0$ be the corresponding $\Gamma_0$-invariant $1$-dimensional representation of $\cW_0$. By \cref{ind same}, $\Ind(M_0)$ corresponds to the quantization $\calD\simeq \Ind(\calD_0)$. Thus, $\Ind (M_0)\simeq M$, and therefore $M$ is induced from $\fl$.
		\end{proof2}
    \subsection{Generalizations of the parabolic induction}
        We finish the main body of the paper with a short summary of results discussed in this paper that allow us to obtain \cref{theorem2}, and whether our methods can be adapted to a more general situation.

        \begin{itemize}
            \item[1)] Results of this paper allow us to uniquely extend a quantization from $X^{reg}$ to $X$ for $X=\Spec(\CC[\OO])$, see \cref{extension}. This is not possible in general, but \cref{extend} implies that it can be done for $X$, such that $\codim_X(X^{sing})\ge 3$. 
            \item[2)] If $X$ admits a Hamiltonian action of Levi subgroup $L\subset G$, the parabolic induction can be generalized to a map from Hamiltonian $L$-equivariant quantizations of $X$ to Hamiltonian $G$-equivariant quantizations of $G\times^P (X\times \fp^\perp)$ by directly adapting arguments of \cref{para quant}.
            \item[3)] Assume that $X$ is a conical symplectic singularity, $G\times^P (X\times \fp^\perp)\to X'=\Spec(\CC[G\times^P (X\times \fp^\perp)])$ is a partial resolution, and $X'$ is a conical symplectic singularity. It can be shown that the parabolic induction produces a bijection between Hamiltonian $L$-equivariant quantizations of $X$ and Hamiltonian $G$-equivariant quantizations of $X'$. This strategy was implemented in \cite{LMBM} for $X=\Spec(\CC[\widetilde{\OO}_0])$, where $\widetilde{\OO}_0$ is an $L$-equivariant cover of nilpotent orbit in $\fl^*$. 
            \item[4)] In terms and notation of the more general situation, \cref{theorem2} shows a bijection on the set of Hamiltonian equivariant quantizations of $X^{reg}$ and $X'^{reg}$, which in the case considered in this paper coincide with $\OO_0$ and $\OO$ respectively. To obtain an analogous result, we need an explicit description of the sets of Hamiltonian equivariant quantizations of $X^{reg}$ and $X'^{reg}$. That can be obtained either by 1), or by the period map of \cite{BK} if $X^{reg}$ and $X'^{reg}$ are admissible in terminology of \emph{loc.cit.} We remark that $\OO_0$ and $\OO$ are not admissible.
            
        \end{itemize}
    \begin{appendices}
	\section{Extension of a holomorphic quantization to a codimension $4$}
	In this section we give a proof of \cref{codim 2}.
	
	\begin{prop}
		Let $X$ be a normal Cohen-Macaulay affine Poisson variety with finitely many symplectic leaves. Let $X_2$ be the union of the smooth part $X^{reg}$ and all symplectic leaves of codimension $2$. Then any (holomorphic) quantization of $X_2$ uniquely extends to a (holomorphic) quantization of $X$. 
	\end{prop}
	\begin{proof}
		That is a well-known fact for the algebraic case, we will prove the proposition for holomorphic quantizations. Let $i: X_2\to X$ be the natural embedding. Let $\calD$ be a holomorphic quantization of $X_2$. We will show that $i_*\calD$ is a holomorphic quantization of $X$. Consider an open subset $Y\subset X$. We set $Y_2=Y\cap X_2$ and define $A=\Gamma(Y,\calH_X)$. We want to show that $\Gamma(Y_2, \calD)=\Gamma(Y, i_*\calD)$ is a quantization of $A$ for every Stein $Y$. Then we have the short exact sequence $0\to \Gamma(Y, \hbar i_*\calD)\to \Gamma(Y,i_*\calD)\to \Gamma(Y, \calH_X)\to 0$ for all Stein open subsets $Y$. Such subsets form a base of topology for $X$, so it implies the short exact sequence of sheaves $0\to \hbar i_*\calD\to i_*\calD\to \calH_X\to 0$, and therefore $i_*\calD$ is a quantization of $X$. So it is enough to show that $\Gamma(Y_2, \calD)$ is a quantization of $A$.
		
		We have a short exact sequence $0\to \hbar\calD\to \calD\to \calH_{Y_2}\to 0$ of sheaves on $Y_2$. Applying the left exact functor $\Gamma$, we get a long exact sequence of cohomology $0\to \Gamma(Y_2,\hbar\calD)\to \Gamma(Y_2,\calD)\to \Gamma(Y_2, \calH_{Y_2})\to H^1(Y_2,\hbar\calD)\to H^1(Y_2, \calD)\to \ldots$. From the Hartogs extension theorem $\Gamma(Y_2, \calH_{Y_2})\simeq A$. We will show that $H^1(Y_2, \calH_{Y_2})=0$ and deduce that the map $H^1(Y_2,\hbar\calD)\to H^1(Y_2,\calD)$ is an isomorphism. First, we need the following lemma.
		\begin{lemma}\label{anCohMac}
			Let $X$ be a Cohen-Macaulay algebraic scheme. Then $X_{hol}$ is a Cohen-Macaulay analytic space.
		\end{lemma}
		\begin{proof}
			Consider a point $x\in X_{hol}$. We need to show that the local ring $\calH_{X,x}$ is Cohen-Macaulay. Let $\widehat{\calH}_{X,x}$ be the completion of $\calH_{X,x}$ with respect to the maximal ideal. By \cite[Theorem 17.5]{matsumura_1987}, $\calH_{X,x}$ is Cohen-Macaulay if and only if $\widehat{\calH}_{X,x}$ is Cohen-Macaulay. The same argument shows that $\widehat{\cO}_{X,x}$ is Cohen-Macaulay. By \cref{isost} $\widehat{\calH}_{X,x}\simeq \widehat{\cO}_{X,x}$. Lemma follows.
		\end{proof}
		
		\begin{lemma}\label{coh_zero}
			$H^1(Y_2, \calH_{Y_2})=0$.
		\end{lemma}
		\begin{proof}
			We set $Z=Y-Y_2$ to be the intersection of the union of all symplectic leaves of codimension greater or equal $4$ with $Y$. We have a long exact sequence of sheaf cohomology 
			
			\begin{equation*}\label{sequence1}
				\ldots\to H^1(Y, \calH_Y)\to H^1(Y_2, \calH_{Y_2})\to H^2_Z(Y, \calH_Y)\to \ldots. \tag{1}
			\end{equation*}  $Y$ is a Stein analytic scheme, so $H^1(Y, \calH_Y)=0$. Analogously to \cite[Theorem 3.8]{Hartshorne1967}, $H^i_Z(Y, \calH_Y)=0$ for all $i<\depth_Z(\calH_Y)$. By \cref{anCohMac} $X_{hol}$ is Cohen-Macaulay, so $\depth_Z(\calH_Y)=\codim Z$ and $H^2_Z(Y, \calH_Y)=0$. The first and third terms in (\ref{sequence1}) are zero, so $H^1(Y_2, \calH_{Y_2})=0$.
		\end{proof}
		
		We have a short exact sequence $0\to \hbar\calD\to \calD\to \calH_{Y_2}\to 0$ on $Y_2$ that induces the long exact sequence of cohomology 
		\begin{equation*}\label{sequence2}
			0\to \Gamma(Y_2,\hbar\calD)\to \Gamma(Y_2,\calD)\to A\to H^1(Y_2, \hbar\calD)\to H^1(Y_2, \calD)\to H^1(Y_2, \calH_{Y_2}) \tag{2}
		\end{equation*} 
		The last term in (\ref{sequence2}) is trivial, so we have a surjective map $H^1(Y_2, \hbar\calD)\to H^1(Y_2, \calD)$. We need the following lemma.
		
		\begin{lemma}\label{Cech}
			Let $\calD$ be a sheaf of complete and separated $\CC[[\hbar]]$-algebras on a compact analytic space $X$. Suppose that the natural map $f: {H}^d(X, \hbar\calD)\to H^d(X, \calD)$ is surjective for some $d$. Then $f$ is an isomorphism.
		\end{lemma}
		\begin{proof}
			We use an argument from \cite[Lemma 5.6.3]{Gordon2014}. Since $X$ is compact, we can compute sheaf cohomology by Cech complex. Consider a Leray cover $\{U_i\}$ of $X$. Let $C^i$, $Z^i$ and $B^i$ be the cochains, cocycles and coboundaries for the Cech cohomology of $\calD$ with respect to $\{U_i\}$. The natural map $f: {\hbar Z^d}/{\hbar B^d}={H}^d(X, \hbar \calD)\to {H}^d(X, \calD)={Z^d}/{B^d}$ is surjective, so $\hbar Z^d+B^d=Z^d$. We want to show that $\hbar Z^d\cap B^d=\hbar B^d$. 
			
			Consider an element $x=\hbar z_0\in \hbar Z^d\cap B^d$. We can decompose $z_0$ into a sum $\hbar z_1+b_1$ for some $z_1\in Z^d$ and $b_1\in B^d$. Then $x=\hbar^2z_1+\hbar b_1$, and we set $x_1=\hbar b_1$. Analogously, we have a decomposition $z_1=\hbar z_2+b_2$, and then $x=\hbar^3z_2+\hbar^2 b_2+\hbar b_1$. We set $x_2=\hbar b_2+\hbar b_1$. Applying the same procedure we get a sequence $x_k$, such that $x=x_k+\hbar^{k+1}z_k$ and $x_k\in \hbar B^d$. Since $Z^d$ is complete and separated, this sequence converges to $x$. Therefore $x\in \hbar B^d$, and $ \hbar Z^d\cap B^d=\hbar B^d$. 
			
			Then $\hbar ({Z^d}/{B^d})={\hbar Z^d}/({B^d\cap \hbar Z^d})={\hbar Z^d}/{\hbar B^d}$, and the natural embedding $\hbar ({Z^d}/{B^d})\to {Z^d}/{B^d}$ is surjective. Therefore $f$ is an isomorphism. 
		\end{proof}
		Applying the lemma above to (\ref{sequence2}) we get that the map $H^1(Y_2, \hbar \calD)\to H^1(Y_2,\calD)$ is an isomorphism, and $\Gamma(Y_2,\calD)$ is a quantization of $A$. 
	\end{proof}

    \end{appendices}
    
    \bmhead{Acknowledgements}
    I am deeply grateful to Ivan Losev for the statement of the problem and for many fruitful discussions.
	I would like to thank Alexey Pakharev for stimulating conversations and many helpful suggestions.
    I am also grateful to the anonymous referee for numerous suggestions that helped to improve the exposition.

    \section*{Declarations}
    \bmhead{Competing interests} 
    The author has no competing interests to declare that are relevant to the content of this article.

     \bmhead{Ethical Approval}
     Not applicable.

     \bmhead{Funding}
     None.
     
     \bmhead{Availability of data and materials}
     Not applicable.
     
    \bibliography{nb-bibliography}
\end{document}